\documentclass[leqno,11pt,a4paper]{amsart}
\usepackage{amsthm}
\usepackage{cite}
\usepackage{enumitem}
\usepackage{hyperref}
\usepackage{graphicx}
\usepackage{amssymb}
\usepackage{xcolor}
\hypersetup{
  colorlinks,
  linkcolor={blue!50!black},
  citecolor={blue!50!black},
  urlcolor={blue!80!black}
}
\usepackage{arydshln}
\usepackage{multirow}
\usepackage[all]{xy}
\usepackage{tikz}
\usepackage{booktabs}
\usepackage{mleftright}

\setlist[enumerate]{labelsep=*, leftmargin=1.5pc}
\setlist[enumerate]{label=\normalfont(\roman*), ref=\roman*}
\usepackage[margin=2.7cm]{geometry}
\graphicspath{{images/}}
\theoremstyle{plain}
\newtheorem{thm}{Theorem}[section]
\newtheorem{pro}[thm]{Proposition}
\newtheorem{lem}[thm]{Lemma}

\theoremstyle{definition}
\newtheorem{dfn}[thm]{Definition}
\newtheorem{rem}[thm]{Remark}
\newtheorem{eg}[thm]{Example}
\newtheorem{algorithm}[thm]{Algorithm}

\DeclareMathOperator{\coeff}{coeff}

\DeclareMathOperator{\Hom}{Hom}

\DeclareMathOperator{\Pic}{Pic}

\DeclareMathOperator{\Div}{Div}

\DeclareMathOperator{\Nef}{\overline{Amp}}

\DeclareMathOperator{\Newt}{Newt}
\DeclareMathOperator{\cone}{cone}

\newcommand{\cA}{\mathcal{A}}
\newcommand{\cO}{\mathcal{O}}
\newcommand{\QQ}{{\mathbb{Q}}}

\newcommand{\PP}{{\mathbb{P}}}

\newcommand{\ZZ}{{\mathbb{Z}}}

\newcommand{\FF}{\mathbb{F}}
\newcommand{\CC}{\mathbb{C}}

\newcommand{\LL}{\mathbb{L}}

\newcommand{\cM}{\mathbb{M}}
\newcommand{\tM}{\widetilde{M}}
\newcommand{\tN}{\widetilde{N}}
\newcommand{\Cstar}{\CC^\times}

\renewcommand{\emptyset}{\varnothing}
\renewcommand{\tilde}{\widetilde}
\newcommand{\conv}[1]{\operatorname{conv}\mleft({#1}\mright)}
\newcommand{\Scan}{S_{\text{\rm can}}}
\newcommand{\V}[1]{\operatorname{verts}\mleft({#1}\mright)}

\begin{document}
\author[T.\,Coates]{Tom Coates}
\address{Department of Mathematics\\Imperial College London\\London, SW$7$\ $2$AZ\\UK}
\email{t.coates@imperial.ac.uk}
\author[A.\,M.\,Kasprzyk]{Alexander Kasprzyk}
\address{School of Mathematical Sciences\\University of Nottingham\\Nottingham, NG$7$\ $2$RD\\UK}
\email{a.m.kasprzyk@nottingham.ac.uk}
\author[T.\,Prince]{Thomas Prince}
\address{Department of Mathematics\\Imperial College London\\London, SW$7$\ $2$AZ\\UK}
\email{t.prince12@imperial.ac.uk}
\title{Laurent Inversion}
\maketitle
\begin{abstract}
  We describe a practical and effective method for reconstructing the deformation class of a Fano manifold~$X$ from a Laurent polynomial~$f$ that corresponds to $X$ under Mirror Symmetry. We explore connections to nef partitions, the smoothing of singular toric varieties, and the construction of embeddings of one (possibly-singular) toric variety in another. In particular, we construct degenerations from Fano manifolds to singular toric varieties; in the toric complete intersection case, these degenerations were constructed previously by Doran--Harder.
 We use our method to find models of orbifold del~Pezzo surfaces as complete intersections and degeneracy loci, and to construct a new four-dimensional Fano manifold.
\end{abstract}
\section{Introduction}
The classification of Fano manifolds is an important open problem in geometry. As things stand the classification is understood only in dimensions one, two, and three~\cite{Iskovskih:1,Iskovskih:2,Iskovskih:anticanonical,Mori--Mukai:Manuscripta,Mori--Mukai:Tokyo,Mori--Mukai:Kinosaki,Mori--Mukai:Erratum,Mori--Mukai:Turin}, but Golyshev et al.~ have announced a new approach to Fano classification~\cite{Golyshev, CCGGK}, using Mirror Symmetry, that could potentially work in all dimensions. Extensive computational experiments suggest that, under Mirror Symmetry, $n$-dimensional Fano manifolds correspond to certain Laurent polynomials in $n$ variables with very special properties. We understand how to recover the known classifications in low dimensions from this perspective~\cite{A+,ACGK,CCGK}, but two essential questions remain: 
\begin{enumerate}[label=(\Alph*),ref=\Alph*] 
\item what is the class of Laurent polynomials $f$ that correspond, under Mirror Symmetry, to Fano manifolds $X$?\label{problem:A} 
\item given such a Laurent polynomial $f$, how can we construct the corresponding $X$?\label{problem:B} 
\end{enumerate}

There has been significant recent progress on Question~\ref{problem:A}: deformation families of Fano manifolds conjecturally correspond to mutation-equivalence classes of certain \emph{rigid maximally mutable Laurent polynomials}~\cite{A+,Kasprzyk--Tveiten}. In this paper we make significant progress on Question~\ref{problem:B}. There are well-understood methods, going back to Givental and Hori--Vafa, that to a Fano toric complete intersection~$X$ associate a Laurent polynomial~$f$ that corresponds to $X$ under Mirror Symmetry. We describe a technique, \emph{Laurent inversion}, for inverting this process, constructing the toric complete intersection~$X$ directly from its Laurent polynomial mirror~$f$. In many cases this allows, given a Laurent polynomial $f$, the direct construction of a Fano manifold $X$ that corresponds to $f$ under Mirror Symmetry. Thus, in many cases, Laurent inversion answers Question~\ref{problem:B}. In fact, as we explain in~\S\ref{sec:beyond_ci}, when phrased appropriately, Laurent inversion is not limited to toric complete intersections: we can use it to construct Fano manifolds $X$ as degeneracy loci (cut out by Pfaffian-type equations), and to give other classical constructions. As proof of concept, in~\S\ref{sec:new_4d} we construct a new four-dimensional Fano manifold by applying Laurent inversion to a rigid maximally-mutable Laurent polynomial in four variables. 

The idea of reconstructing a Fano manifold $X$ from its mirror $f$ is not new. It is expected that, if a Fano manifold $X$ is mirror to a Laurent polynomial~$f$, then there is a degeneration from $X$ to the (singular) toric variety $X_f$ defined by the spanning fan of the Newton polytope of $f$; such a degeneration has been constructed for complete intersections in partial flag manifolds by Doran--Harder~\cite{Doran--Harder}. Thus one might hope to recover the Fano manifold $X$ from~$f$ by smoothing $X_f$, for instance using the Gross--Siebert program\footnote{This works in dimension two~\cite{Prince}, but the higher-dimensional case is  significantly more involved.}~\cite{Gross--Siebert}, or via deformation theory~\cite{A95,A97,CI14,CI16,I12}. Our new contribution here is to give an explicit construction of $X$, rather than a proof of its existence. Indeed, regardless of its context, Laurent inversion gives a powerful new method for constructing algebraic varieties. We illustrate this in~\S\ref{sec:new_third_one_one} below, where we exhibit explicit models for del~Pezzo surfaces with $1/3(1,1)$ singularities that played an essential role in the Corti--Heuberger classification~\cite{CH}, and which are hard to construct using more traditional methods.

As we will see in~\S\ref{sec:laurent_inversion}, in many cases Laurent inversion constructs, along with $X$, an embedded degeneration from $X$ to the singular toric variety $X_f$ -- thus implementing the expected smoothing of $X_f$ discussed above. We hope therefore that Laurent inversion will give a substantial hint as to the generalisations required to get a Gross--Siebert-style smoothing procedure working in higher dimensions. In the toric complete intersection case, such an embedded degeneration has been constructed by Doran--Harder~\cite{Doran--Harder}; we build an explicit link to their work in~\S\ref{sec:torus_charts}, where we describe how our main combinatorial construction, \emph{scaffolding}, can be seen as a generalisation of their notion of amenable collection. We also discuss (in~\S\ref{sec:nef_partitions}) how scaffolding gives a generalisation to the Fano case of Borisov's celebrated \emph{nef partitions}, which have proved a powerful tool for constructing mirror partners to Calabi--Yau toric complete intersections~\cite{B93,BB96}. It will be very interesting to see how much of the theory survives to the Calabi--Yau case, and whether we can use Laurent inversion to construct and investigate Calabi--Yau manifolds that are not complete intersections.

\section{Laurent Polynomial Mirrors for Toric Complete Intersections}
\label{sec:forward}

We begin by recalling how to associate to a toric complete intersection $X$ a Laurent polynomial that corresponds to $X$ under Mirror Symmetry. This question has been considered by many authors~\cite{Givental:toric,Hori--Vafa,Prz:1,Prz:2,Doran--Harder,CKP}, and we will give a construction which generalises and unifies all these perspectives below (in~\S\ref{sec:torus_charts}). Consider first the ambient toric variety or toric stack $Y$. We consider the case where:
\begin{equation}
\label{eq:ambient_conditions}
  \begin{minipage}{0.94\linewidth}
    \begin{enumerate}
    \item $Y$ is a proper toric Deligne--Mumford stack;  
    \item the coarse moduli space of $Y$ is projective;  
    \item the generic isotropy group of $Y$ is trivial, that is, $Y$ is a toric \emph{orbifold}; and  
    \item at least one torus-fixed point in $Y$ is smooth. 
    \end{enumerate}
  \end{minipage}
\end{equation}
Conditions~(i)--(iii) here are essential; condition~(iv) is less important and will be removed in~\S\ref{sec:torus_charts}.
In the original work by Borisov--Chen--Smith~\cite{Borisov--Chen--Smith}, toric Deligne--Mumford stacks are defined in terms of stacky fans. In our context, since the generic isotropy is trivial, giving a stacky fan that defines $Y$ amounts to giving a triple $(N;\Sigma;\rho_1,\ldots,\rho_R)$ where $N$ is a lattice, $\Sigma$ is a rational simplicial fan in $N \otimes \QQ$, and $\rho_1,\ldots,\rho_R$ are elements of $N$ that generate the rays of $\Sigma$. It will be more convenient for our purposes, however, to represent $Y$ as a GIT quotient $\big[ \CC^R/\!\!/_{\!\omega} (\Cstar)^r\big]$. Any such $Y$ can be realised this way, as we now explain.

\begin{dfn}\label{dfn:GIT_data}
  We say that $(K;\LL;D_1,\ldots,D_R;\omega)$ are \emph{GIT data} if $K \cong (\Cstar)^r$ is a connected torus of rank~$r$; $\LL = \Hom(\Cstar,K)$ is the lattice of subgroups of $K$; $D_1,\ldots,D_R \in \LL^*$ are characters of $K$ that span a strictly convex full-dimensional cone in $\LL^* \otimes \QQ$, and $\omega \in \LL^* \otimes \QQ$ lies in this cone.
\end{dfn}

GIT data $(K;\LL;D_1,\ldots,D_R;\omega)$ determine a quotient stack $\big[ V_\omega/K \big]$ with $V_\omega \subset \CC^R$, as follows. The characters $D_1,\ldots,D_R$ define an action of $K$ on $\CC^R$. For convenience write $[R] := \{1,2,\ldots,R\}$. We say that a subset $I \subset [R]$ \emph{covers} $\omega$ if and only if $\omega = \sum_{i \in I} a_i D_i$ for some strictly positive rational numbers $a_i$. Set $\cA_\omega = \{ I \subset [R]\mid \text{$I$ covers $\omega$}\}$, and set
\begin{align*}
  V_\omega = \bigcup_{I \in \cA_\omega} (\Cstar)^I \times \CC^{\bar{I}} && \text{where} && (\Cstar)^I \times \CC^{\bar{I}} = \big\{(x_1,\ldots,x_R) \in \CC^R\mid \text{$x_i \ne 0$ if $i \in I$}\big\}. 
\end{align*}
The subset $V_\omega \subset \CC^R$ is $K$-invariant, and $\big[ V_\omega/K\big]$ is the GIT quotient stack given by the action of $K$ on $\CC^R$ and the stability condition $\omega$. The  convexity hypothesis in Definition~\ref{dfn:GIT_data} ensures that $\big[ V_\omega/K\big]$ is proper.

\begin{rem}
  Recall that the quotient $\big[ V_\omega/K\big]$ depends on $\omega$ only via the minimal cone $\sigma$ of the secondary fan such that $\omega \in \sigma$. The \emph{secondary fan} for $(K;\LL;D_1,\ldots,D_R;\omega)$ is the fan defined by the wall-and-chamber decomposition of the cone in $\LL^* \otimes \QQ$ spanned by $D_1,\ldots,D_R$, where the walls are given by all $(r-1)$-dimensional cones of the form $\{D_i\mid i \in I\}$ with $I \subset [R]$.
\end{rem}

\begin{dfn}
  \emph{Orbifold GIT data} are those such that the quotient $\big[ V_\omega/K\big]$ is a toric orbifold.
\end{dfn}

The quotient $\big[ V_\omega/K\big]$ is a toric Deligne--Mumford stack if and only if $\omega$ lies in the strict interior of a maximal cone in the secondary fan. A toric orbifold $Y$ satisfying conditions~\eqref{eq:ambient_conditions} above arises as the quotient $\big[ V_\omega/K\big]$ for GIT data $(K;\LL;D_1,\ldots,D_R;\omega)$ as follows. Suppose that $Y$ is defined by the stacky fan data $(N;\Sigma;\rho_1,\ldots,\rho_R)$. There is an exact sequence
\begin{equation}
\label{eq:fan_sequence}
  \xymatrix{
    0 \ar[r] & \LL \ar[r] & \ZZ^R \ar[r]^\rho  & N \ar[r] & 0
  }
\end{equation}
where $\rho$ maps the $i$th element of the standard basis for $\ZZ^R$ to $\rho_i$; this defines $\LL$ and $K = \LL \otimes \Cstar$. Dualising gives 
\begin{equation}
\label{eq:divisor_sequence}
  \xymatrix{
    0 & \LL^* \ar[l] & (\ZZ^*)^R \ar[l]_D  & M \ar[l] & 0 \ar[l]
  }
\end{equation}
where $M := \Hom(N,\ZZ)$, and we set $D_i \in \LL^*$ to be the image under $D$ of the $i$th standard basis element for $(\ZZ^*)^R$. The stability condition $\omega$ is taken to lie in the strict interior of 
\[
C := \!\!\bigcap_{\text{maximal cones $\sigma$ of $\Sigma$}}\!\!C_\sigma
\]
where $C_\sigma$ is the cone in $\LL^* \otimes \QQ$ spanned by $\{D_i\mid i \not \in \sigma\}$; projectivity of the coarse moduli space of $Y$ implies that $C$ is a maximal cone of the secondary fan, and in particular that $C$ has non-empty interior.

We can reverse this construction, defining a stacky fan $(N;\Sigma;\rho_1,\ldots,\rho_n)$ from GIT data $(K;\LL;D_1,\ldots,D_R;\omega)$ such that $D_1,\ldots,D_R$ span $\LL^*$. The lattice $\LL$ and elements $D_1,\ldots,D_R \in \LL^*$ define the exact sequence~\eqref{eq:divisor_sequence}, and dualising gives~\eqref{eq:fan_sequence}. This defines the lattice $N$ and $\rho_1,\ldots,\rho_R$. The fan $\Sigma$ consists of the cones spanned by $\{\rho_i\mid i \in I\}$ where $I \subset [R]$ satisfies $[R]\setminus I \in \cA_\omega$.

\begin{rem}
  Once $K$,~$\LL$, and~$D_1,\ldots,D_R$ have been fixed, choosing $\omega$ such that the GIT data $(K;\LL;D_1,\ldots,D_R;\omega)$ define a toric Deligne--Mumford stack amounts to choosing a maximal cone in the secondary fan. 
\end{rem}

\begin{rem}
  A character $\chi \in \LL^*$ determines a line bundle on $Y$, which we denote also by $\chi$.
\end{rem}

\begin{dfn}\label{dfn:convex_partition_with_basis}
  Let $\Theta = (K;\LL;D_1,\ldots,D_R;\omega)$ be orbifold GIT data, and let $Y$ denote the corresponding toric orbifold. A \emph{convex partition with basis} for $\Theta$ is a partition $B,S_1,\ldots,S_k,U$ of $[R]$
  such that:
  \begin{enumerate}
  \item $\{D_b\mid b \in B\}$ is a basis for $\LL^*$; 
  \item $\omega$ is a non-negative linear combination of $\{D_b\mid b \in B\}$;
  \item each $S_i$ is non-empty; 
  \item for each $i \in [k]$, the line bundle $L_i := \sum_{j \in S_i} D_j$ on $Y$ is convex\footnote{A line bundle $L$ on a Deligne--Mumford stack $Y$ is convex if and only if $L$ is nef and is the pullback of a line bundle on the coarse moduli space $|Y|$ of $Y$ along the structure map $Y \to |Y|$. See~\cite{CGIJJM}.}; and 
  \item for each $i \in [k]$, $L_i$ is a non-negative linear combination of $\{D_b\mid b \in B\}$. 
  \end{enumerate}
  We allow $k=0$, and we allow $U=\emptyset$. 
\end{dfn}

\begin{rem}
  Since $\omega$ here is taken to lie in the strict interior of a maximal cone in the secondary fan, it is given by a positive linear combination of $\{D_b\mid b \in B\}$. This positivity guarantees that the maximal cone spanned by $\{\rho_i\mid i \in [R]\setminus B\}$ defines a smooth torus-fixed point in $Y$.
\end{rem}

\begin{rem}
  It would be more natural to replace the condition that $L_i$ be convex here with the weaker condition that $L_i$ be nef. But, since we currently lack a Mirror Theorem that applies to toric complete intersections beyond the convex case, we will require convexity. If the ambient space $Y$ is a manifold, rather than an orbifold, then a line bundle on $Y$ is convex if and only if it is nef.
\end{rem}

Given:
\begin{equation}
\label{eq:lots_of_data}
  \begin{minipage}{0.94\linewidth}
    \begin{enumerate}
    \item orbifold GIT data $\Theta = (K;\LL;D_1,\ldots,D_R;\omega)$;
    \item a convex partition with basis $B, S_1, \ldots, S_k, U$ for $\Theta$; and
    \item a choice of elements $s_i \in S_i$ for each $i \in [k]$;
    \end{enumerate}
  \end{minipage}
\end{equation}
we define a Laurent polynomial $f$ as follows. This is the \emph{Przyjalkowski method}; cf.~\cite[\S5]{CKP}. Without loss of generality we may assume that $B = [r]$. Writing $D_1,\ldots,D_R$ in terms of the basis $\{D_b\mid b \in B\}$ for $\LL^*$ yields an $r \times R$ matrix $\cM = (m_{i,j})$ of the form
\begin{equation}
\label{eq:weight_matrix}
  \cM = \left(
    \begin{array}{c:ccc}
      \multirow{3}{*}{$\quad I_r\quad$} & m_{1,r+1} & \cdots & m_{1,R} \\
      & \vdots & & \vdots \\
      & m_{r,r+1} & \cdots & m_{r,R} \\
    \end{array}
  \right)
\end{equation}
where $I_r$ is an $r \times r$ identity matrix. Consider the function
\[
W = x_1 + x_2 + \cdots + x_R - k
\]
subject to the constraints
\begin{align}
\label{eq:ambient_constraints}
  \prod_{j=1}^R x_j^{m_{i,j}} = 1 && 1\leq i\leq r, \\
  \intertext{and}
\label{eq:hypersurface_constraints}
  \sum_{j \in S_i} x_j = 1 && 1\leq i\leq k.
\end{align}
For each $i \in U$, introduce a new variable $y_i$. For each $i \in [k]$, introduce new variables $y_j$, where $j \in S_i \setminus \{s_i\}$, and set $y_{s_i} = 1$. Solve the constraints~\eqref{eq:hypersurface_constraints} by setting:
\begin{align*}
  &x_j = \frac{y_j}{\sum_{l \in S_i} y_l} & j \in S_i, \\
  &x_j = y_j & j \in U,
\end{align*}
and express the variables $x_b$, $b \in B$, in terms of the $y_j$s using~\eqref{eq:ambient_constraints}. The function $W$ thus becomes a Laurent polynomial~$f$ in the variables $y_j$, $j \in [R] \setminus \{s_1,\ldots,s_k\}$. We refer to $y_j$, $j \in U$, as \emph{uneliminated variables}. 

Given data as in~\eqref{eq:lots_of_data}, let $f$ be the Laurent polynomial just defined. Let $Y$ denote the toric orbifold determined by $\Theta$, let $L_1,\ldots,L_k$ denote the line bundles on $Y$ from Definition~\ref{dfn:convex_partition_with_basis}, and let $X \subset Y$ be a complete intersection defined by a regular section of the vector bundle $\oplus_i L_i$. If $X$ is Fano, then Mirror Theorems due to Givental, Hori--Vafa, and others~\cite{Givental:toric,Hori--Vafa,CCIT:1,CCIT:2} imply that $f$ corresponds to $X$ under Mirror Symmetry; c.f.~\cite[\S5]{CKP}. We say that $f$ \emph{is a Laurent polynomial mirror} for $X$.

\begin{rem}\label{rem:translation}
  If $f$ is a Laurent polynomial mirror for $X$ then the Picard--Fuchs local system for $f \colon (\Cstar)^n \to \CC$ coincides, after translation of the base if necessary, with the Fourier--Laplace transform of the quantum local system for~$X$; see~\cite{CCGGK,CCGK}. Thus we regard $f$ and $g := f-c$, where $c$ is a constant, as Laurent polynomial mirrors for the same manifold $Y$, since the Picard--Fuchs local systems for $f$ and $g$ differ only by a translation of the base (by~$c$).
\end{rem}

\begin{rem}\label{rem:reparametrization}
  If $f$ and $g$ are Laurent polynomials that differ by an invertible monomial change of variables then the Picard--Fuchs local systems for~$f$ and~$g$ coincide. Thus $f$ is a Laurent polynomial mirror for $X$ if and only if $g$ is a Laurent polynomial mirror for $X$.
\end{rem}

\begin{eg}\label{eg:cubic_forwards}
  Let $X$ be a smooth cubic surface. The ambient toric variety $Y = \PP^3$ is a GIT quotient $\CC^4 /\!\!/ \Cstar$ where $\Cstar$ acts on $\CC^4$ with weights $(1,1,1,1)$. Thus $Y$ is given by GIT data $(K;\LL;D_1,\ldots,D_4;\omega)$ with $K=\Cstar$, $\LL = \ZZ$, $D_1=D_2=D_3=D_4=1$, and $\omega=1$. We consider the convex partition with basis $B$,~$S_1$,~$\emptyset$, where $B = \{1\}$ and $S_1 = \{2,3,4\}$, and take $s_1 = 4$. This yields
  \[
  \cM = \left( 
    \begin{array}{c:ccc}
      1&1&1&1
    \end{array}
  \right)
  \]
  and
  \[
  W = x_1 + x_2 + x_3 + x_4 - 1
  \]
  subject to 
  \begin{align*}
    x_1 x_2 x_3 x_4 = 1 && \text{and} && x_2 + x_3 + x_4 = 1.
  \end{align*}
  We set:
  \begin{align*}
    x_1 = \frac{1}{x_2 x_3 x_4}, &&
    x_2 = \frac{x}{1+x+y}, &&                                  
    x_3 = \frac{y}{1+x+y}, &&                                  
    x_4 = \frac{1}{1+x+y}, &&                                  
  \end{align*}
  where, in the notation above, $x = y_2$ and $y = y_3$. Thus
  \[
  f =  \frac{(1+x+y)^3}{xy}
  \]
  is a Laurent polynomial mirror to $Y$.
\end{eg}

\begin{eg}
  Let $Y$ be the projective bundle $\PP\big(\cO \oplus \cO \oplus \cO(-1)\big) \to \PP^3$. This arises from the GIT data $(K;\LL;D_1,\ldots,D_7;\omega)$ where $K = (\Cstar)^2$, $\LL = \ZZ^2$, 
  \begin{align*}
    D_1 = D_4 = D_6 = D_7 = (1,0), && D_2 = D_3 = (0,1), && D_5 = (-1,1),
  \end{align*}
  and $\omega = (1,1)$. We consider the convex partition with basis $B,S_1,S_2,U$ where $B = \{1,2\}$, $S_1 = \{3,4\}$, $S_2 = \{5,6\}$, $U = \{7\}$. This yields:
  \[
  \cM =
  \left(
  \begin{array}{cc:ccccc}
    1&0&0&1&-1&1&1\\
    0&1&1&0&1&0&0
  \end{array}
  \right)
  \]
  Choosing $s_1 = 3$ and $s_2 = 5$, we find that
  \[
  f = \frac{(1+x)}{xyz} + (1+x)(1+y) + z
  \]
  Here, in the notation above, $x = y_4$, $y = y_6$, and $z = y_7$.
\end{eg}

\section{Scaffolding}

In this section we give our central combinatorial construction: that of \emph{scaffolding}. The output from the Przyjalkowski method is a Laurent polynomial $f$ together with a decomposition of $f$ as a sum of terms $x_i$, each of which is a Laurent polynomial in the variables $y_j$. The Newton polytope of each of the terms $x_i$ is a product of translated dilates of standard simplices. Therefore each $\Newt(x_i)$ is the polyhedron $P_D$ of sections of a nef divisor $D$ on some (fixed) product of projective spaces. This motivates the following definition.

\begin{dfn}
\label{dfn:scaffolding}
Fix the following data:
\begin{enumerate}
\item a lattice $N$ together with a splitting $N = \bar{N} \oplus N_U$;
\item the dual lattice $M := \Hom(N,\ZZ)$, with the dual splitting $M = \bar{M} \oplus M_U$;
\item a Fano polytope $P \subset N_\QQ$;
\item a projective toric variety $Z$ given by a fan in $\bar{M}$ whose rays span the lattice $\bar{M}$.
\end{enumerate}
Given such data, a \emph{scaffolding} $S$ of $P$ is a set of pairs $(D,\chi)$ where $D$ is a nef divisor on $Z$ and $\chi$ is an element of $N_U$, such that
\[
P = \conv{P_D + \chi \, \Big| \, (D,\chi) \in S}.
\]
We refer to $Z$ as the \emph{shape} of the scaffolding, and the elements $(D,\chi) \in S$ as \emph{struts}.
\end{dfn}

\begin{lem}
\label{lem:scaffolding_from_poly}
Let $f$ be a Laurent polynomial produced using the Przyjalkowksi method in~\S\ref{sec:forward}. The polytopes $\Newt(x_i)$ determine a scaffolding of $P = \Newt(f)$ such that the shape $Z$ is the product of projective spaces
\[
Z := \PP^{|S_1|-1}\times\cdots\times\PP^{|S_k|-1}
\] 
and $S$ contains $r + |U|$ struts.
\end{lem}
\begin{proof}
The polytope $P$ is the convex hull of the union of the polytopes $\Newt(x_i)$ for $x_i$ not appearing in any of the equations~\eqref{eq:hypersurface_constraints}. There is a splitting of $N$ into the sublattice $N_U$ spanned by the exponents of uneliminated variables $y_j$, $j \in U$, and the sublattice $\bar{N}$ spanned by the exponents of variables $y_i$, $i \not \in U$. If $y_j$ is an uneliminated variable, add the strut $(\cO,\Newt(y_j))$ to $S$. For $i \not \in U$, $\Newt(x_i)$ is the polyhedron of sections of a nef divisor $D$ on $Z$, translated by an element $\chi \in N_U$, and we add the strut $(D,\chi)$ to $S$. By construction $P$ is the convex hull of this collection of struts. 
\end{proof}

\begin{rem}\label{rem:uneliminated_basis}
Note that any scaffolding generated by the proof of Lemma~\ref{lem:scaffolding_from_poly} contains a collection of struts $\{(\cO,e_i) \mid i \in I\}$ for an index set $I$, corresponding to uneliminated variables, such that the collection $\{e_i \mid i \in I\}$ forms a basis of $N_U$. Although not the most general setting possible, we will assume from here onwards that this condition holds for every scaffolding.
\end{rem}

Using the shape $Z$ we can phrase the `inversion' technique as a double application of Mirror Symmetry. Going forwards we start from a complete intersection $X \subset Y$ and form a Laurent polynomial $f$. The scaffolding obtained in the proof of Lemma~\ref{lem:scaffolding_from_poly} expresses $f$ as a sum of sections of nef divisors on $Z$. Going backwards, the Givental/Hori--Vafa mirror of $Z$ is a torus fibration $Z^\vee$ together with a regular function $W$ on $Z^\vee$. The nef divisors we found to describe $f$ determine the compactifying boundary divisors of $Z^\vee \subset Y$.

\begin{eg}[$dP_3$]\label{ex:dP3_start}
Consider the Laurent polynomial
\[
f=\frac{(1+x+y)^3}{xy}
\]
from Example~\ref{eg:cubic_forwards}. A scaffolding for $\Newt(f)$ is given by a single standard 2-simplex, dilated by a factor of three:
\begin{center}
\includegraphics{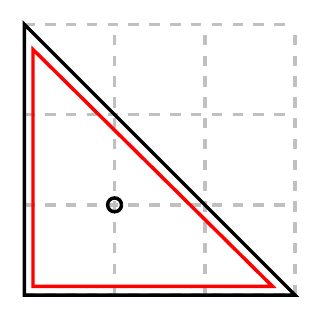}
\end{center}
This gives a scaffolding of $\Newt(f)$ by single strut, with no uneliminated variables. The shape $Z$ is $\PP^2$ and the strut is given by choosing the entire toric boundary of $\PP^2$.
\end{eg}

\begin{eg}[$dP_6$]
\label{ex:dP6_foreshadow}
Consider the Laurent polynomial
\[
f=x+y+\frac{1}{x}+\frac{1}{y}+\frac{x}{y}+\frac{y}{x}.
\]
This is a mirror to the del~Pezzo surface $dP_6$. We may scaffold $\Newt(f)$ in two different ways, using either three triangles or a pair of squares:
\begin{center}
\includegraphics{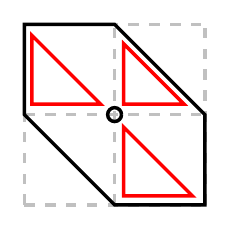}
\raisebox{30px}{$\quad\text{ and }\quad$}
\includegraphics{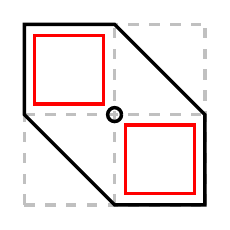}
\end{center}
These choices correspond, respectively, to the decompositions
\[
f=(1+x+y) +\frac{(1+x+y)}{x}+\frac{(1+x+y)}{y} - 3
\ \text{ and }\ 
f=\frac{(1+x)(1+y)}{x}+\frac{(1+x)(1+y)}{y} - 2.
\]
As discussed in Remark~\ref{rem:translation}, we ignore the constant terms.
\end{eg}

\section{A Dual Perspective on Scaffolding}\label{sec:dual_perspective}

There is a dual characterisation of scaffolding which is often useful in applications. Instead of considering the polytope $P$, we consider the cone $C(P^*)$ over the dual polytope $P^*$  embedded at height one in $M_\QQ \oplus \QQ$, and interpret the struts of a scaffolding as certain cones whose common intersection is exactly $C(P^*)$.

\begin{dfn}
Given a Fano polytope $P$, let $C(P^*)$ be the cone obtained by embedding the rational polytope $P^*$ in $M_\QQ \oplus \{1\}$ and forming the cone over this affine polytope. Given a scaffolding $S$ of $P$ and a strut $s = (D, \chi)$ in $S$, define $C_s$ to be the cone
\[
C_s := \left\{ (\bar{m},u,z) \in \big(\bar{M}\oplus M_U\big)_\QQ \oplus \QQ \mid  z \geq \phi_D(\bar{m}) + \chi(u)\right\} \subset M_\QQ \oplus \QQ
\]
where $\phi_D$ is the piecewise linear function on $\bar{M}$ determined by the $\QQ$-Cartier divisor $D$ on $Z$.
\end{dfn}

\begin{rem}
Recall that a torus invariant Weil divisor $D \in \Div_{T_{\bar{M}}}(Z)$ is, by definition, an integer-valued function on the set of rays of the fan $\Sigma_Z$ determined by $Z$. The divisor $D$ is $\QQ$-Cartier if and only if this function is realised by a piecewise linear function $ \phi_D$ on the fan of $Z$. Moreover the divisor $D$ is nef if and only if the function $\phi_D$ is convex. The polyhedron of sections $P_D$ of the divisor $D$ is defined as the intersection of half-spaces $\langle \rho,- \rangle \geq -\phi_D(\rho)$ where $\rho$ ranges over the integral generators of the rays of $\Sigma_Z$. Thus the rays of the cone $C_s$ are generated by pairs $(\rho,k)$ where $k = (\phi_D-\chi)(\rho)$ is the height of the supporting hyperplane of the strut $P_D+\chi$.
\end{rem}

\noindent We can now interpret $S$ as a collection of cones whose mutual intersection is equal to $C(P^*)$.

\begin{lem}
Given data as in (i)--(iv) of Definition~\ref{dfn:scaffolding} and a collection $S$ of pairs $s = (D, \chi)$, where $D$ is a nef divisor on $Z$ and $\chi \in N_U$, then $S$ is a scaffolding if and only if
\[
\bigcap_{s \in S}{C_s} = C(P^*).
\]
\end{lem}
\begin{proof}
  Given a pair $s = (D,\chi) \in \Nef(Z)\times N_U$ we prove that $C(P^*) \subseteq C_s$ if and only if the strut $P_D + \chi \subset P$. Since $D$ is nef, $C_s$ is a convex cone and so without loss of generality we can replace the condition that $C(P^*) \subset C_s$ with the condition that each of the rays of $C(P^*)$ is contained in $C_s$. Fixing a ray of $C(P^*)$ generated by an element $\rho \in M\oplus \ZZ$, recall that $\rho = (\rho',1)$ where $\rho'$ is a vertex of $P^*$. Considering the family of hyperplanes $H_{\rho',r} := \big\{ n \in N_\QQ \mid \langle\rho', n\rangle = r\big\}$,~$r \in \QQ$, we see that $-1$ is the minimal $r$ such that $H_{\rho',r}$ meets $P^*$ and that the minimal value of $r$ such that $H_{\rho',r}$ meets $P_D+\chi$ is $-(\phi_D-\chi)(\rho')$. Thus $P_D + \chi \subset P$ if and only if $-(\phi_D-\chi)(\rho') \geq -1$ for all $\rho'$.

It remains to show that equality holds for the inclusion
\[
C(P^*) \subseteq \bigcap_{s \in S}{C_s}
\]
precisely when $S$ is a scaffolding. In other words we need to show that the equality $C(P^*) = \bigcap_{s \in S}C_s$ is equivalent to the condition that
\[
P = \conv{P_D + \chi \, \Big| \, (D,\chi) \in S}.
\]
If $P$ is the convex hull of the polytopes $P_D + \chi$ then every vertex of $P$ meets a strut $P_D+\chi$. In that case every facet of $C(P^*)$ is contained in a facet of some $C_s$ and so, in particular, the intersection of the cones $C_s$ is contained in the cone $C(P^*)$. Conversely if the intersection of cones $C_s$ is equal to $C(P^*)$ then every ray $\langle (\rho',1) \rangle$ of $C(P^*)$ is contained in some $C_s$, and therefore the minimal $r \in \QQ$ such that $H_{\rho',r}$ meets \emph{some} polytope $P_D+\chi$ is equal to $-1$.
\end{proof}

\begin{eg}  
\label{eg:dP6_before}
\begin{figure}[htb!]
  \centering
  \includegraphics[width=\textwidth]{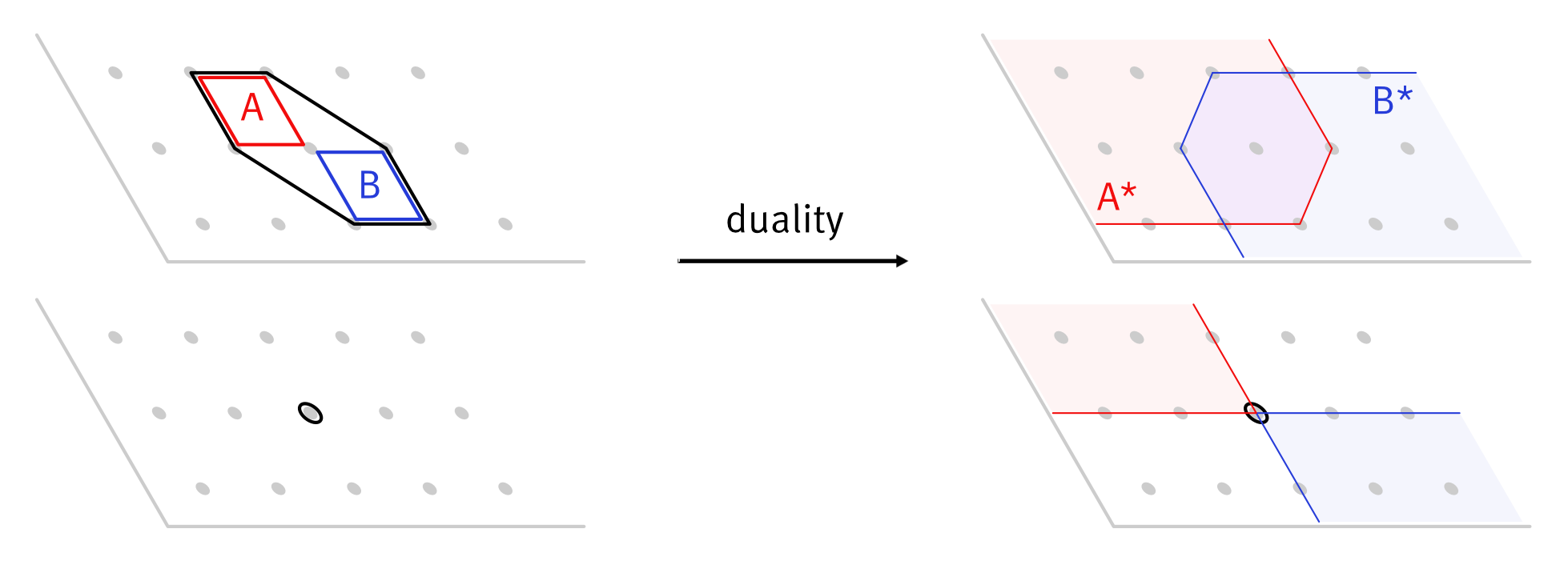}
  \caption{The dual picture of one of the scaffoldings from Example~\ref{ex:dP6_foreshadow}.}
\label{fig:dP6_before}
\end{figure}

Consider the right-hand scaffolding in Example~\ref{ex:dP6_foreshadow}. This is shown again on the left-hand side of Figure~\ref{fig:dP6_before}, placed at height~$1$ in $N_\QQ \oplus \QQ$ with the struts labelled as $A$ and~$B$. The corresponding cones $C_A$ and $C_B$ in $M_\QQ\oplus \QQ$ are shown on the right-hand side of Figure~\ref{fig:dP6_before}: $C_A$ is the cone over the dual polyhedron $A^*$, placed at height~$1$ in $M_\QQ\oplus \QQ$, and similarly for $C_B$. The tail cones $T_{A^*}$ of $A^*$ and $T_{B^*}$ of $B^*$ are shown at height zero: these are faces of $C_A = C(A^*) = C(A)^\vee$ and $C_B = C(B^*) = C(B)^\vee$ respectively. The shape $Z$ can be recovered by projecting the facets of $C_A$ and $C_B$ onto the height-zero slice in $M_\QQ \oplus \QQ$; this gives the fan of $Z = \PP^1 \times \PP^1$. The heights of the rays of $C_A$ (respectively $C_B$) determine a divisor $D_A = D_1 + D_2$ (respectively $D_B = D_3 + D_4$) on $Z$. The strut $A$ can be recovered as the polytope of sections of $\cO(D_A)$, and similarly for $B$.
\end{eg}

Note that in this dual perspective it makes sense to relax the condition that the divisors $D$ of struts $s = (D,\chi)$ be nef on $Z$. Indeed, the new definition of scaffolding makes sense so long as $D$ is $\QQ$-Cartier, the cost of which is that the cones $C_s$ cease to be convex. (Recall that the convexity of $C_s$ is equivalent to $D$ being a nef divisor.)  Whilst we do not explore this further here, we hope that this notion will prove useful in the study of polytopes up to mutation.

\section{Laurent Inversion}
\label{sec:laurent_inversion}

We have seen that if $X$ is a Fano toric complete intersection defined by convex line bundles $L_1,\ldots, L_k$ on a toric orbifold $Y$, then there is a Laurent polynomial mirror $f$ for $X$ and a decomposition 
\begin{equation}
\label{eq:scaffolding}
  f =  f_1 + \cdots + f_r + \sum_{u \in U} x_u
\end{equation}
where
\[
f_a = \prod_{i=1}^k \prod_{j \in S_i} \left( \frac{\sum_{l \in S_i} y_l}{y_j} \right)^{m_{a,j}} \times \prod_{u \in U} x_u^{-m_{a,u}}.
\]
This decomposition of $f$ determines GIT data $(K;\LL;D_1,\ldots,D_R)$ for $Y$, except for the stability condition, and also the line bundles $L_1,\ldots,L_k$. Indeed all of this data can be recovered from the scaffolding $S$ of $\Newt(f)$ given by Lemma~\ref{lem:scaffolding_from_poly}. In this section we generalise this observation, describing how to pass from a scaffolding $S$ of a Fano polytope $P$ to a toric variety $Y$ and a toric embedding $X_P \to Y$.

\begin{algorithm}\label{alg:laurent_inversion}
Let $S$ be a scaffolding of a Fano polytope $P$ with shape $Z$. Let $u = \dim N_U$ and let $r = |S| - u$, so that $S$ contains $r$ struts that do not correspond to uneliminated variables and $u$ struts that do correspond to uneliminated variables (see Remark~\ref{rem:uneliminated_basis}). Let $R$ be the sum of $|S|$ and the number $z$ of rays of $Z$. We determine an $ r \times R$ matrix $\cM$, which will be the weight matrix for our toric variety $Y$, as follows. Let $m_{i,j}$ denote the $(i,j)$ entry of $\cM$. Fix an identification of the rows of $\cM$ with the $r$ elements $(D_i,\chi_i)$ of $S$ which do not correspond to uneliminated variables, and an ordering $\Delta_1,\ldots,\Delta_z$ of the toric divisors in $Z$. Let $e_1,\ldots,e_u$ be the basis of $N_U$ given by Remark~\ref{rem:uneliminated_basis}.
\begin{enumerate}
\item For $1 \leq j \leq r$ and any $i$, let $m_{i,j} = \delta_{i,j}$.
\item For $1 \leq j \leq u$ and any $i$, let $m_{i,r+j}$ be determined by the expansion 
\[
\chi_i = \sum_{j=1}^u{m_{i,r+j}e_j}.
\]
\item For $1 \leq j \leq z$, let $m_{i,|S|+j}$ be determined by the expansion
\[
D_i = \sum_{j=1}^z{m_{i,|S|+j}}\Delta_j.
\]
\end{enumerate}
The weight matrix $\cM$ alone does not determine a unique toric variety -- we also need to choose a stability condition $\omega$. Let $Y_\omega$ denote the toric variety determined by this choice. Unless otherwise stated, we will take $\omega$ to be the sum of the first $|S|$ columns in $\cM$. 
\end{algorithm}

\begin{rem}
  In terms of the dual perspective on scaffoldings in~\S\ref{sec:dual_perspective}, the entry $m_{i,|S|+j}$ in the matrix $\cM$ is the height in $M_\QQ \oplus \QQ$ of the $j$th ray in the $i$th cone $C_s$.
\end{rem}

\begin{rem}
  In the case where the scaffolding $S$ arises from a toric complete intersection $X$ via Lemma~\ref{lem:scaffolding_from_poly}, the choice of $\omega$ given above is equal to $-K_X-\sum_{i \in [k]}L_i$. The corresponding convex partition with basis $B$, $S_1$, \ldots, $S_k$, $U$ can be recovered by setting $B = \{1,2,\ldots,r\}$, $U = \{r+1,\ldots, r+u\}$, and $S_j$ equal to the subset of $\{|S|+1,\ldots,|S|+z\}$ given by the toric divisors on the $j$th factor $\PP^{a_j}$ of $Z = \prod_{i=1}^k \PP^{a_i}$.
\end{rem}

\begin{rem}
The ray lattice of $Y_\omega$, that is, the lattice of one-parameter subgroups of the dense torus in $Y_\omega$, is equal to
$\Div_{T_{\bar{M}}}(Z)\oplus N_U$. 
\end{rem}

In favourable cases, a suitable choice of stability condition $\omega$ gives a smooth toric orbifold $Y_\omega$ and convex line bundles $L_1,\ldots,L_k$ on $Y_\omega$ such that the complete intersection $X \subset Y_\omega$ defined by a regular section of the vector bundle $\oplus_i L_i$ is Fano. This can be very useful, and we use it in~\S\ref{sec:new_4d} to exhibit a new four dimensional Fano manifold. However our construction is not restricted to the case where the scaffolding comes from a toric complete intersection via Givental/Hori--Vafa mirror symmetry; that is, we do not insist that the shape $Z$ is a product of projective spaces. In the Appendix we prove:

\begin{thm}
\label{thm:embedding}
A scaffolding $S$ of $P$ determines a torus invariant embedding of $X_P$ into $Y_\omega$.
\end{thm}

\noindent Thus \emph{any} scaffolding of a Fano polytope $P$ determines a toric embedding of the corresponding Fano toric variety $X_P$ into an ambient toric variety. If the scaffolding $S$ arises, via Lemma~\ref{lem:scaffolding_from_poly}, from a Fano toric complete intersection $X$ defined by convex line bundles $L_1,\ldots,L_k$ on a Fano toric orbifold $Y$, then Theorem~\ref{thm:embedding} embeds $X_P$ as a complete intersection in a toric variety $Y_\omega$ defined using the same GIT data as $Y$ (but with a possibly-different stability condition $\omega$); see~\S\ref{sec:embedded_ci}. There is then often an embedded degeneration from $X$ to $X_P$. In general, however, the embedding in Theorem~\ref{thm:embedding} is not a complete intersection, and $X_P$ may not have an embedded smoothing inside~$Y_\omega$. Example~\ref{ex:Pf} is instructive here.

The map of tori in Theorem~\ref{thm:embedding}, of which the embedding  $X_P \hookrightarrow Y_\omega$ is the closure in $Y_\omega$, is as follows. The dense tori in $X_P$ and $Y_\omega$ are $T_N$ and $T_{\widetilde{N}}$ respectively. There is a map 
\[
\bar{N} \oplus N_U = N \rightarrow \tilde{N} = \Div_{T_{\bar{M}}}(Z)\oplus N_U
\] 
defined on each factor as:
\begin{enumerate}
\item $\bar{N} \rightarrow \Div_{T_{\bar{M}}}(Z) \oplus \{0\}$, the map taking characters of $T_{\bar{M}}$ to principal divisors;
\item $N_U \rightarrow \{0\} \oplus N_U$, the identity map.
\end{enumerate}

\noindent For example, if $Z$ is a product of projective spaces then the ray map dualises to give an inclusion of tori $T_N \hookrightarrow T_{\widetilde{N}}$ with ideal generated by binomials of the form $(\prod{x_i} = 1)$, where the product is taken over variables corresponding to divisors in the same projective space factor.

\section{Examples}

In this section we apply Algorithm~\ref{alg:laurent_inversion} to several concrete examples.

\begin{eg}[$dP_3$]
Continuing Example~\ref{ex:dP3_start}, recall the scaffolding obtained from a mirror to $dP_3$ given by a single standard 2-simplex, dilated by a factor of three:
\begin{center}
\includegraphics{dP3}
\end{center}
From this we read off $u=0$, $r=1$, $R=4$, $B = \{1\}$, $U = \varnothing$, $S_1 = \{2,3,4\}$, and
\[
\cM = \left( 
  \begin{array}{c:ccc}
    1&1&1&1
  \end{array}
\right).
\]
This gives GIT data $\Theta = (K;\LL;D_1,\ldots,D_4;\omega)$ with $K=\Cstar$, $\LL = \ZZ$, $D_1=D_2=D_3=D_4=1$, and $\omega=1$; note that the secondary fan here has a unique maximal cone.
The corresponding toric variety is $Y = \PP^3$. The ideal defining $X_P$ is principal in Cox co-ordinates on $Y$, generated by the equation $X_1X_2X_3 - X^3_0$. This is a section of the nef line bundle $\cO(3)$. Thus $B,S_1,\varnothing$ is a convex partition with basis for $\Theta$, and we obtain the cubic hypersurface as in Example~\ref{eg:cubic_forwards}.
\end{eg}

\begin{eg}[$dP_6$]
\label{ex:dP6}
The projective plane blown up in three points, $dP_6$, is toric, but it has two famous models as a complete intersection:
\begin{enumerate}
\item as a hypersurface of type $(1,1,1)$ in $\PP^1\times\PP^1\times\PP^1$;
\item as the intersection of two bilinear equations in $\PP^2\times\PP^2$.
\end{enumerate}
Recall the two scaffoldings from Example~\ref{ex:dP6_foreshadow}, which arose from the two decompositions
\[
f=(1+x+y) +\frac{(1+x+y)}{x}+\frac{(1+x+y)}{y} - 3
\ \text{ and }\ 
f=\frac{(1+x)(1+y)}{x}+\frac{(1+x)(1+y)}{y} - 2
\]
of a Laurent polynomial mirror $f$ for $dP_6$.

From the first scaffolding we read off $u=0$, $r=3$, $Z=\PP^2$, $R=6$, $B = \{1,2,3\}$,  $U = \varnothing$, $S_1 = \{4,5,6\}$, and
\[
\cM = \left(
\begin{array}{ccc:ccc}
  1 & 0 & 0 & 1 & 0 & 0 \\
  0 & 1 & 0 & 0 & 1 & 0 \\
  0 & 0 & 1 & 0 & 0 & 1
\end{array}
\right).
\]
This gives GIT data $\Theta = (K;\LL;D_1,\ldots,D_6;\omega)$ with $K=(\Cstar)^3$, $\LL = \ZZ^3$, $D_1=D_4=(1,0,0)$, $D_2=D_5=(0,1,0)$, $D_3=D_6=(0,0,1)$, and $\omega=(1,1,1)$; the secondary fan here again has a unique maximal cone. The corresponding toric variety is $Y = \PP^1 \times \PP^1 \times \PP^1$. The line bundle $L_1 = \sum_{j \in S_1} D_j$ is $\cO(1,1,1)$, so we see that $f$ is a Laurent polynomial mirror to a hypersurface of type $(1,1,1)$ in $\PP^1 \times \PP^1 \times \PP^1$.

From the second scaffolding we read off $u=0$, $r=2$, $Z = \PP^1 \times \PP^1$, $B=\{1,2\}$, $U = \varnothing$, $S_1 = \{3,4\}$, $S_2 = \{5,6\}$, and 
\[
\cM = \left(
\begin{array}{cc:cccc}
1&0&0&1&1&0 \\
0&1&1&0&0&1 
\end{array}
\right).
\]
This gives GIT data $\Theta = (K;\LL;D_1,\ldots,D_6;\omega)$ with $K=(\Cstar)^2$, $\LL = \ZZ^2$, $D_1=D_4=D_5=(1,0)$, $D_2=D_3=D_6=(0,1)$, and $\omega=(1,1)$; once again the secondary fan has a unique maximal cone. The corresponding toric variety $Y$ is $\PP^2 \times \PP^2$. The line bundles $L_1 = D_3 + D_4$ and $L_2 = D_5 + D_6$ are both equal to $\cO(1,1)$, so we see that $f$ is a Laurent polynomial mirror to the complete intersection of two hypersurfaces defined by bilinear equations in $\PP^2 \times \PP^2$.

\end{eg}

\begin{eg}[{$\mathrm{MM}_\text{3--4}$}]\label{eg:3-4}
  \begin{figure}[ht]
    \centering
    \includegraphics[scale=0.6]{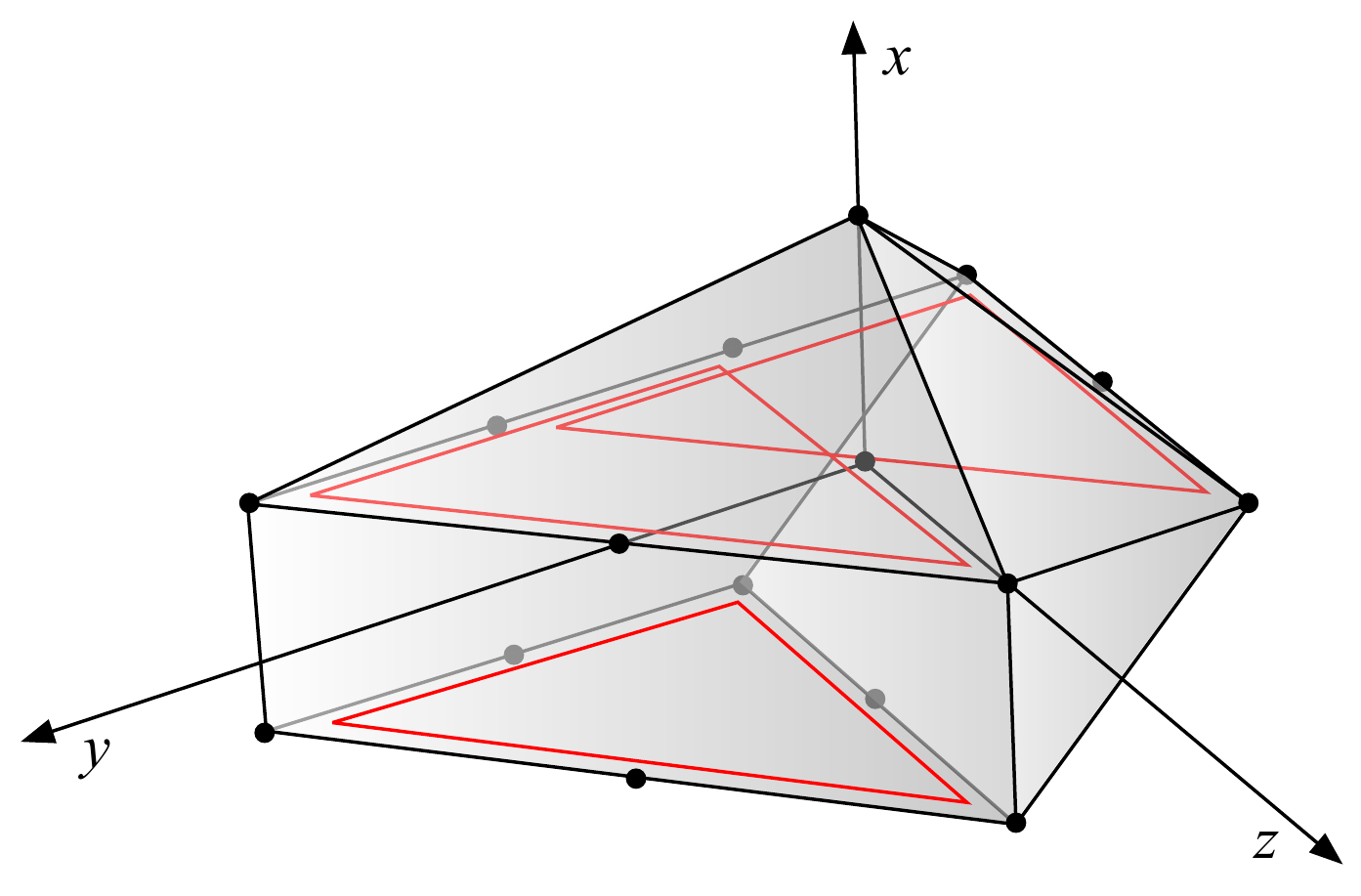}
    \caption{A scaffolding for $\Newt(f)$ in Example~\ref{eg:3-4}.}
\label{fig:3-4}
  \end{figure}
  Consider the rigid maximally-mutable Laurent polynomial
  \[
  f = x + \frac{y^2}{z} + 2 y + \frac{3 y}{z} + z + \frac{3}{z} + \frac{z}{y} + \frac{2}{y} + \frac{1}{y z} + \frac{y^2}{x z} + \frac{2 y}{x} + \frac{2 y}{x z} + \frac{z}{x} + \frac{2}{x} + \frac{1}{x z}.
  \]
  The Newton polytope of $f$ can be scaffolded as in Figure~\ref{fig:3-4}, and there is a corresponding decomposition of $f$:
  \[
  f = x + \frac{(1+y+z)^2}{xz} + \frac{(1+y+z)^2}{z} + \frac{(1+y+z)^2}{yz}
  \]
  From this we read off $u=1$, $r=3$, $Z = \PP^2$, $B=\{1,2,3\}$, $U=\{4\}$, $S_1 = \{5,6,7\}$,  and 
  \[
  \cM = \left(
    \begin{array}{ccc:cccc}
      1 & 0 & 0 & 1 & 1 & 0 & 1 \\
      0 & 1 & 0 & 0 & 1 & 0 & 1\\
      0 & 0 & 1 & 0 & 0 & 1 & 1
    \end{array}
  \right).
  \]
  This gives GIT data $\Theta = (K;\LL;D_1,\ldots,D_6;\omega)$ with $K=(\Cstar)^3$, $\LL = \ZZ^3$, $D_1=D_4=(1,0,0)$, $D_2=(0,1,0)$, $D_3=D_6=(0,0,1)$, $D_4=(1,1,0)$, and $D_7=(1,1,1)$.
  The secondary fan is as shown in Figure~\ref{fig:3-4_secondary_fan}. Choosing $\omega = (3,2,1)$ yields a weak Fano toric manifold $Y_\omega$ such that the line bundle $L_1 = \sum_{j \in S_1} D_j$ is convex. Let $X$ denote the hypersurface in $Y$ defined by a regular section of $L_1$. The class $-K_Y - L_1$ is nef but not ample on $Y$, but it becomes ample on restriction to $X$; thus $X$ is Fano (cf.~\cite[\S 57]{CCGK}). We see that $f$ is a Laurent polynomial mirror to $X$. This example shows that our Laurent inversion technique applies in cases where the ambient space $Y$ is not Fano. In fact $Y$ need not even be weak Fano.

  \begin{figure}[hbtp]
    \centering
    \begin{tikzpicture}[scale=0.5,arr/.style={thin,->,shorten >=2pt,shorten <=2pt,>=stealth}]
      \coordinate [label={below left:\tiny $(0,0,1)$}] (A) at (0, 0);
      \coordinate [label={above left:\tiny $(0,1,0)$}] (B) at (0, 6);
      \coordinate [label={below right:\tiny $(1,0,0)$}] (C) at (6,0);
      \coordinate [label={left:\tiny $(1,1,1)$}] (D) at (2,2);
      \coordinate [label={above right:\tiny $(1,1,0)$}] (E) at (3,3);
      
      \draw [thick] (A) -- (B) -- (E) -- (C) -- (A);
      \draw [thick] (A) -- (D) -- (E);
      \draw [thick] (B) -- (D) -- (C);
      
      \node (L) at (3.2,5.15) {\tiny $L_1$};
      \node (mK) at (5.55,2.2) {\tiny ${-K_Y}$};
      
      \coordinate (K) at (12/5,9/5);
      
      \draw [fill=black] (K) circle (0.1);
      \draw [fill=black] (2,2) circle (0.1);
      
      \draw[arr] (3,5) to [bend right=15] (2,2) ;
      \draw[arr] (5.3,2) to [bend left=5] (K) ;
    \end{tikzpicture}
    \caption{The secondary fan for Example~\ref{eg:3-4}, sliced by the plane $x+y+z=1$.}
\label{fig:3-4_secondary_fan}
  \end{figure}
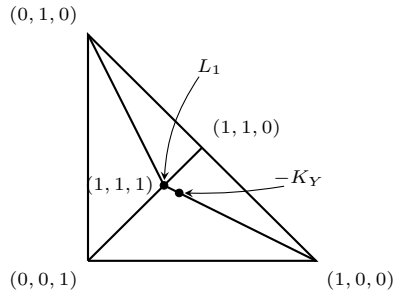
\end{eg}

\section{A New Four-Dimensional Fano Manifold}\label{sec:new_4d}

Consider 
\[
f =  x + y + z + \frac{(1+w)^2}{xzw} + \frac{w}{y}
\]
This is a rigid maximally-mutable Laurent polynomial in four variables. It is presented in scaffolded form, and we read off
$r=2$, $u=3$, $B=\{1,2\}$,  $U=\{3,4,5\}$, $S_1 = \{6,7\}$, and
\[
\cM = \left(
  \begin{array}{cc:ccccc}
    1 & 0  & 1 & 0 & 1 & 1 & 1 \\
    0 & 1  & 0 & 1 & 0 & 1 & -1
  \end{array}
\right).
\]
This yields GIT data $\Theta = (K;\LL;D_1,\ldots,D_6;\omega)$ with $K=(\Cstar)^2$, $\LL = \ZZ^2$, $D_1=D_3=D_5=(1,0)$, $D_2=D_4=(0,1)$, $D_6=(1,1)$, and $D_7=(1,-1)$. We choose the stability condition $\omega = (3,2)$, thus obtaining a Fano toric orbifold $Y$ such that the line bundle $L_1 = D_6+D_7$ on $Y$ is convex. Let $X$ denote the four-dimensional Fano manifold defined inside $Y$ by a regular section of $L_1$. 

The Fano manifold $X$ is new. To see this, we can compute the regularised quantum period $\widehat{G}_X$ of $X$. Since $f$ is a Laurent polynomial mirror to $X$, the regularised quantum period $\widehat{G}_X$ coincides with the classical period of $f$:
\begin{align*}
  \pi_f(t) = \sum_{d=0}^\infty c_d t^d && \text{where} && c_d = \coeff_1 \big(f^d\big).
\end{align*}
This is explained in detail in~\cite{CCGGK,CCGK}. In the case at hand, 
\[
\widehat{G}_X = \pi_f(t) = 1 + 12 t^3 + 120 t^5 + 540 t^6 + 20160 t^8 +  33600 t^9 + \cdots
\]
and we see that $\widehat{G}_X$ is not contained in the list of regularised quantum periods of known four-dimensional Fano manifolds~\cite{CKP,known_4d}. Thus $X$ is new. We did not find $X$ in our systematic search for four-dimensional Fano toric complete intersections~\cite{CKP}, because there we considered only ambient spaces that are Fano toric \emph{manifolds} whereas the ambient space $Y$ here has non-trivial orbifold structure. This is striking because the degree $K_X^4 = 433$ of $X$ is not that low -- compare this with~\cite[Figure~5]{CKP}. In dimensions~two and~three only Fano manifolds of low degree fail to occur as complete intersections in toric manifolds. The space $Y$ can be obtained as the unique non-trivial flip of the projective bundle $\PP\big(\cO(-1) \oplus \cO^{\oplus 3} \oplus \cO(1)\big)$ over $\PP^1$. As was pointed out to us by Casagrande, the other extremal contraction of $Y$, which is small, exhibits $X$ as the blow-up of $\PP^4$ in a plane conic. This suggests that restricting to smooth ambient spaces when searching for Fano toric complete intersections may omit many Fano manifolds with simple classical constructions.

\section{Scaffoldings and Embedded Degenerations of Complete Intersections}\label{sec:embedded_ci}

We next explain how, if $P$ admits a scaffolding for which the shape $Z$ is a product of projective spaces, $X_P$ can be embedded as the common zero locus of a collection of sections of linear systems on $Y$. In this case $X_P$ is a flat degeneration of the zero locus $X$ of a general section. This embedded degeneration is often a smoothing of $X_P$. It was discovered independently by Doran--Harder~\cite{Doran--Harder}: see~\S\ref{sec:torus_charts} for an alternative view on their construction. 

By assumption we have, as in~\S\ref{sec:forward}, an $r \times R$ matrix $\cM = (m_{i,j})$ of the form:
\[
\cM = \left(
\begin{array}{c:ccc}
  \multirow{3}{*}{$\quad I_r\quad$} & m_{1,r+1} & \cdots & m_{1,R} \\
  & \vdots & & \vdots \\
  & m_{r,r+1} & \cdots & m_{r,R} \\
\end{array}
\right)
\]
such that $l_{b,i} := \sum_{j \in S_i} m_{b,j}$ is non-negative for all $b \in [r]$ and $i \in [k]$. The exact sequence~\eqref{eq:fan_sequence} becomes
\[
\xymatrix{
  0 \ar[r] & \ZZ^r \ar[r]^{\cM^T} & \ZZ^R \ar[r]^\rho  & \widetilde{N} \ar[r] & 0
}
\]
and, writing $\rho_i \in \widetilde{N}$ for the image under $\rho$ of the $i$th standard basis vector in $\ZZ^R$, we find that $\{ \rho_i\mid r<i\leq R\}$ is a distinguished basis for $\widetilde{N}$ and that
\begin{align*}
  \rho_i = -\sum_{j=r+1}^R m_{i,j} \rho_j && \text{ for all }1\leq i\leq r.
\end{align*}
Let $\widetilde{M} = \Hom\big(\widetilde{N},\ZZ\big)$ and define $u_j\in \widetilde{M}$, $1\leq j\leq k$, by
\[
u_j(\rho_i)=\begin{cases}
0,&\text{ if }r < i \leq R\text{ and }i\not\in S_j;\\
1,&\text{ if }r < i \leq R\text{ and }i\in S_j.\\
\end{cases}
\]
Let $N':=\widetilde{N}\cap H_{u_1}\cap\ldots\cap H_{u_k}$ be the sublattice of $\widetilde{N}$ given by restricting to the intersection of the hyperplanes $H_{u_i}:=\{v\in N\mid u_i(v)=0\}$. Let $\Sigma'$ denote the fan defined by intersecting $\Sigma$ with $N'_\QQ$, and let $X'$ be the toric variety defined by $\Sigma'$.

\begin{lem}
\label{lem:ci_case}
The lattice $N'$ is the image of $N$ under the map dual to the ray map of $Z$.
\end{lem}
\begin{proof}
The lattice $N'$ is defined as the vanishing of a collection of elements of the dual lattice $\widetilde{M}$. Since these intersect transversely we have that $\dim N' = \dim N$. To check that $N \subset N'$ we check that each $u_i$ vanishes on $N$. But the vectors $u_i$ form a basis of the kernel of the ray map of $Z$ dual to the inclusion of $\bar{N} \hookrightarrow \Div_{T_{\bar{M}}}(Z)$.
\end{proof}

\noindent Thus $X' = X_P$, and we have embedded $X_P$ in $Y$ as the common zero locus of sections of linear systems defined by the hyperplanes $H_{u_i}$.

\section{Beyond Complete Intersections}\label{sec:beyond_ci}

Any Laurent polynomial obtained from the Givental/Hori--Vafa model gives a scaffolding with shape $Z$ equal to a product of projective spaces (Lemma~\ref{lem:scaffolding_from_poly}) but the definition of scaffolding allows for much more general choices of $Z$. We now show how certain classical constructions appear via scaffolding. For example, for any reflexive polytope $P$ there is a distinguished choice of scaffolding $\Scan$ with shape $Z$ given by a toric crepant terminal $\QQ$-factorialisation of the toric variety defined by the normal fan of $P$, and a single strut covering all of $P$.

\begin{pro}\label{pro:anticanonical}
The embedding $X_P \hookrightarrow Y \cong \PP^{\rho-1}$ determined by the scaffolding $\Scan$ of $P$ is the anticanonical embedding of $X_P$, where $\rho$ is the number of integral points of $P^*$.
\end{pro}
\begin{proof}
That $Y \cong \PP^{\rho-1}$ follows from the definition of polar polytope: the nef divisor of $Z$ used to cover $P$ as a single strut is precisely the toric boundary of $Z$. Indeed every torus invariant section of $-K_{X_P}$ defines a character of $T_N$ which in turn generates a ray of $Z$. The map of tori $T_N \hookrightarrow \CC^\star{}^\rho$ defining the embedding in Theorem~\ref{thm:embedding} is precisely the map of tori defined by these characters of $T_N$.
\end{proof}

\begin{eg}[$dP_7$]
\begin{figure}[hbt]
\includegraphics[angle = 90]{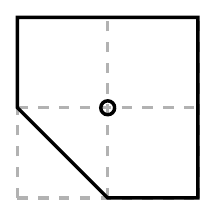}
\caption{Polygon for $dP_7$}
\label{fig:dP7}
\end{figure}
Let $P$ be the polytope shown in Figure~\ref{fig:dP7} and let $Z$ be the toric variety associated to the normal fan of $P$, that is, the blow up of $\PP^2$ in two points. The image of the anticanonical embedding of $X_P$ is the closure in the projective space $\PP^5$ of the variety $X_0$ defined via the following five equations in $\CC^5$:
\begin{align*}
x_1x_3 = 1, &&
x_2x_4 = x_3, &&
x_3x_5 = x_4, &&
x_4x_1 = x_5, &&
x_5x_2 = 1.
\end{align*}
The variety $X_0$ admits a flat deformation to the variety $X_t$ defined by the $4\times 4$ Pfaffians of the following skew-symmetric matrix:
\begin{equation}
\begin{pmatrix}
\phantom{0} & 1 & x_1 & x_2 & 1 \\
& & t & x_3 & x_4 \\
& & & 1 & x_5 \\
& & & & t \\
& & & & \\
\end{pmatrix}
\label{eq:Pfaffians}
\end{equation}
\end{eg}

Scaffoldings of a Fano polygon $P$ using this shape $Z$ produce ambient toric varieties $Y$ which exhibit $X_P$ as the closure in $Y$ of the affine variety defined by these five binomial equations, homogenising each equation to an equation in Cox co-ordinates. In forthcoming work we will show that the existence of the flat deformation of $X_P$ in $Y$ given by these Pfaffians exists if and only if the following `mutability condition' holds. 

\begin{pro}
\label{prop:mutable_pfaffians}
Given a scaffolding $S$ of $X_P$ with shape $Z$, $X_P$ deforms in the ambient space $Y$ to a variety defined by the homogenisation of the $4\times 4$ Pfaffians of~\eqref{eq:Pfaffians} if and only if each strut in $S$, regarded as a polyhedron in $N$, admits mutations, in the sense of~\cite{ACGK}, with weight vectors equal to the elements $x_1$,~$x_2$ (regarded as elements of the dual lattice $M$).
\end{pro}

\noindent This condition ensures that we can homogenise the Pfaffian equations, replacing the entries on the superdiagonal and in the upper-right corner of~\eqref{eq:Pfaffians} with monomials in Cox co-ordinates with \emph{non-negative} exponents.

It follows that $Y$ contains five toric degenerations of the variety defined by these Pfaffians, since cyclically permuting the positions of the variables $x_1, \ldots, x_5$ shown in the matrix~\eqref{eq:Pfaffians} gives rise to five distinct toric degenerations.

\section{Models of Orbifold del~Pezzo Surfaces}
\label{sec:new_third_one_one}

Scaffolding has a practical advantage even in the surface case. In this section we show how to find models of del~Pezzo surfaces with $1/3(1,1)$ singularities that were used by Corti--Heuberger in their classification~\cite{CH}. Two of these models are toric complete intersections; the third is a degeneracy locus cut out by Pfaffian equations. The Fano polygons we use for these models were classified in~\cite{KNP15}. Following~\cite{CH,A+,KNP15} we refer to the del~Pezzo surface with $n \times 1/3(1,1)$ singular points and degree $d$ as $X_{n,d}$.

\begin{eg}[$X_{2,5/3}$]
Consider the Fano polygon $P$ with scaffolding shown in Figure~\ref{fig:X253}. This scaffolding defines the weight matrix:
\[
  \begin{array}{ccccc}
  y_1 & y_2 & x_1 & x_2 & x_3 \\ \midrule
    1 & 0 & 2 & 1 & 1  \\
    0 & 1 & 1 & 2 & -1
  \end{array}
\]
Fixing the stability condition $\omega = (1,1)$ defines a toric variety $Y$. The toric variety $X_P$ is a hypersurface in $Y$ defined by the vanishing of the binomial section $y_1^4y_2^2 - x_1x_2x_3$ of the bundle $L = (4,2)$. A general section of $L$ is a del~Pezzo surface with $2 \times 1/3(1,1)$ singularities and no other singular points.
\begin{figure}[hbtp]
 \includegraphics{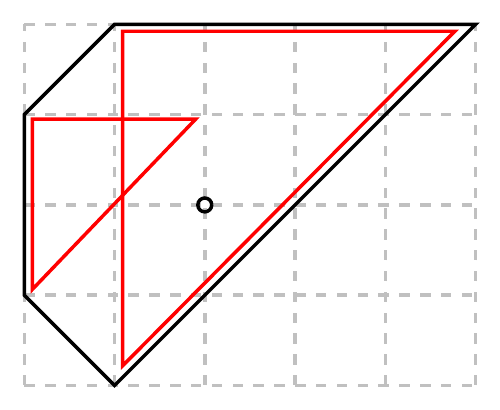}
\caption{Polygon for a degeneration of the surface $X_{2,5/3}$}
\label{fig:X253}
\end{figure}
\end{eg}

\begin{eg}[$X_{3,1}$]
Now consider the Fano polygon $P$ with the scaffolding shown in Figure~\ref{fig:X31}. This scaffolding defines the weight matrix:
\[
  \begin{array}{cccccc}
  y_1 & y_2 & y_3 & x_1 & x_2 & x_3 \\
\midrule
    1 & 0 & 0 & 2 & 1 & 1 \\
    0 & 1 & 0 & 1 & 2 & 1 \\
    0 & 0 & 1 & 1 & 1 & 2 \\
  \end{array}
\]
Fixing the stability condition $\omega = (1,1,1)$ defines a toric variety $Y$. The toric variety $X_P$ is a hypersurface in $Y$ defined by the vanishing of the binomial section $y_1^4y_2^4y_3^4 - x_1x_2x_3$ of the bundle $L = (4,4,4)$. A general section of $L$ is a del~Pezzo surface with $3 \times 1/3(1,1)$ singularities and no other singular points.

\begin{figure}[htpb]
\includegraphics{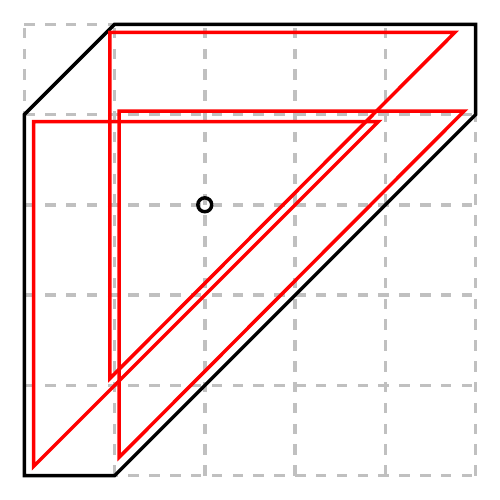}
\caption{Polygon for a degeneration of the surface $X_{3,1}$}
\label{fig:X31}
\end{figure}
\end{eg}

\begin{eg}[$X_{5,5/3}$]\label{ex:Pf}
The surface $X_{5,5/3}$ in~\cite{CH} is found as a degeneracy locus defined by five $4 \times 4$ Pfaffian equations. We show how this appears as an instance of Laurent inversion.

\begin{figure}
\includegraphics{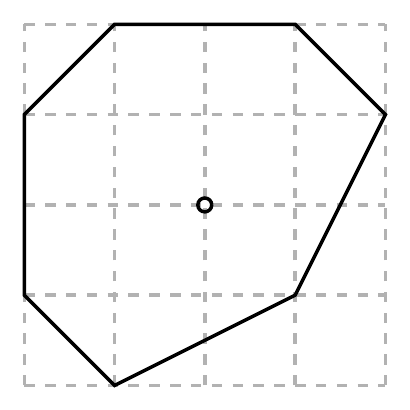}
\caption{Polygon for a degeneration of the surface $X_{5,5/3}$}
\label{fig:X553}
\end{figure}

\begin{figure}
\includegraphics{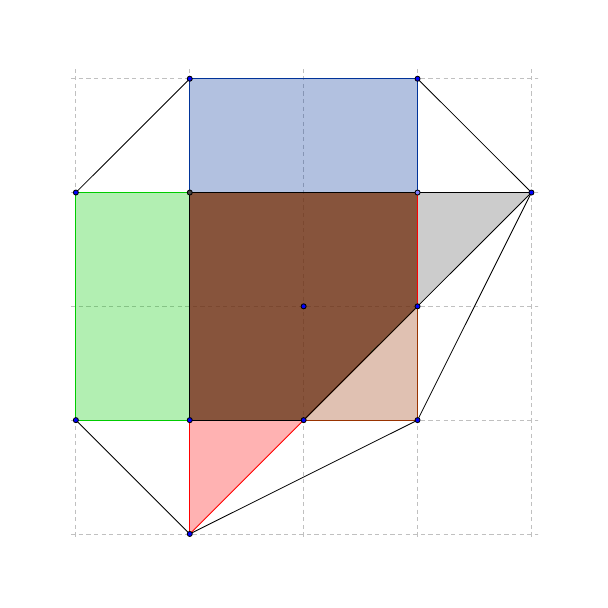}
\caption{}
\label{fig:scaff}
\end{figure}

Consider the polygon $P$ shown in Figure~\ref{fig:X553} and let $Z$ be the toric variety with fan given by the normal fan of the polygon in Figure~\ref{fig:dP7}. Figure~\ref{fig:scaff} exhibits a scaffolding of $P$ with 5 struts and shape $Z$.
To write out the corresponding weight matrix $\cM$ we first have a $5 \times 5$ identity block, identifying each row with a strut; the remaining columns are found by expanding the nef divisors making up the scaffolding in the basis of torus invariant divisors. In this way we obtain:
\[
\begin{matrix}
y_1 & y_2 & y_3 & y_4 & y_5 & x_1 & x_2 & x_3 & x_4 & x_5 \\
\midrule
1 & 0 & 0 & 0 & 0 & 2 & 1 & 1 & 1 & 1 \\
0 & 1 & 0 & 0 & 0 & 1 & 2 & 1 & 1 & 1 \\
0 & 0 & 1 & 0 & 0 & 1 & 1 & 2 & 1 & 1 \\
0 & 0 & 0 & 1 & 0 & 1 & 1 & 1 & 2 & 1 \\
0 & 0 & 0 & 0 & 1 & 1 & 1 & 1 & 1 & 2
\end{matrix}.
\]
This scaffolding satisfies the mutability condition in Proposition~\ref{prop:mutable_pfaffians}. Taking stability condition $(1,\ldots,1)$ and homogenising the Pfaffian equations given in~\eqref{eq:Pfaffians} we obtain a flat deformation of $X_P$ given by the $4 \times 4$ Pfaffians of the skew-symmetric matrix:
\begin{equation}
\begin{pmatrix}
\phantom{0} & y^2_1y_2y_3y^2_4 & x_1 & x_2 & y_1y^2_2y^2_4y_5 \\
& & y^2_1y^2_3y_4y_5 & x_3 & x_4 \\
& & & y^2_1y^2_2y_3y_5 & x_5 \\
& & & & y^2_2y_3y_4y^2_5 \\
& & & & \\
\end{pmatrix}
\label{eq:homo_Pfaffians}
\end{equation}
Hence we realise the surface $X_{5,5/3}$ as a degeneracy locus in a rank 5 toric variety $Y$. In this example all five toric degenerations of the surface are isomorphic. This is not typical, but a consequence of the symmetries of the Fano polygon $P$.
\end{eg}

\section{Nef Partitions}
\label{sec:nef_partitions}

We now consider the connection between Laurent inversion and the \emph{nef partitions} studied by Batyrev and Borisov~\cite{B93,BB96}. We begin with a motivating example. The notion of mutation of polytopes~\cite{ACGK} extends naturally to scaffoldings, and we illustrate this by mutating one of the scaffoldings considered in Example~\ref{ex:dP6}. 

\begin{figure}[htpb]
  \centering
  \begin{tikzpicture}[transform shape]
    \begin{scope}
      \clip (-1.3,-1.3) rectangle (1.3cm,1.3cm); 
      \draw[line width=0.5mm, black] (-1,1) -- (0,1) -- (1,0) -- (1,-1) -- (0,-1) -- (-1,0) -- (-1,1); 
      \draw[line width=0.5mm, red] (-0.9,0.9) -- (-0.1,0.9) -- (-0.1,0.1) -- (-0.9,0.1) -- (-0.9,0.9) -- (-0.1,0.9); 
      \draw[line width=0.5mm, red] (0.1,-0.1) -- (0.9,-0.1) -- (0.9,-0.9) -- (0.1,-0.9) -- (0.1,-0.1) -- (0.9,-0.1); 
      \foreach \x in {-7,-6,...,7}{                        
        \foreach \y in {-7,-6,...,7}{                      
          \node[draw,shape = circle,inner sep=1pt,fill] at (\x,\y) {}; 
        }
      }
      \node[draw,shape = circle,inner sep=4pt] at (0,0) {}; 
    \end{scope}
  \end{tikzpicture}
  \qquad \qquad \qquad
  \begin{tikzpicture}[transform shape]
    \begin{scope}
      \clip (-1.3,-1.3) rectangle (1.3cm,1.3cm); 
      \draw[line width=0.5mm, black] (-1,1) -- (1,1) -- (1,0) -- (0,-1) -- (-1,0) -- (-1,1); 
      \draw[line width=0.5mm, red] (-0.9,0.9) -- (0.9,0.9) -- (0.1,0.1) -- (-0.9,0.1) -- (-0.9,0.9) -- (0.9,0.9); 
      \draw[line width=0.5mm, red] (0,-0.1) -- (0.8,-0.1) -- (0,-0.9) -- (0,-0.1) -- (0.8,-0.1); 
      \foreach \x in {-7,-6,...,7}{                        
        \foreach \y in {-7,-6,...,7}{                      
          \node[draw,shape = circle,inner sep=1pt,fill] at (\x,\y) {}; 
        }
      }
      \node[draw,shape = circle,inner sep=4pt] at (0,0) {}; 
    \end{scope}
  \end{tikzpicture}
  
  \caption{Mutating a scaffolding}
\label{fig:mutating_scaffolding}
\end{figure}
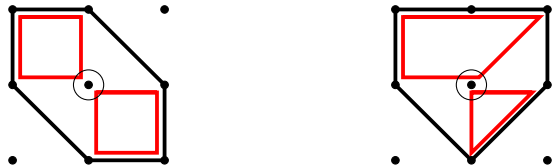

\begin{eg}
\label{ex:dP6_mutated}

The mutation that takes the left-hand polygon in Figure~\ref{fig:mutating_scaffolding} (previously seen in Example~\ref{ex:dP6}) to the right-hand polygon transforms the scaffolding as shown.
In Example~\ref{eg:dP6_before} we analysed the dual picture of the left-hand scaffolding in Figure~\ref{fig:mutating_scaffolding}, obtaining Figure~\ref{fig:dP6_before_again} (which is a copy of Figure~\ref{fig:dP6_before}). Repeating this analysis for the right-hand scaffolding in Figure~\ref{fig:mutating_scaffolding}

\begin{figure}[htb!]
  \centering
  \includegraphics[width=\textwidth]{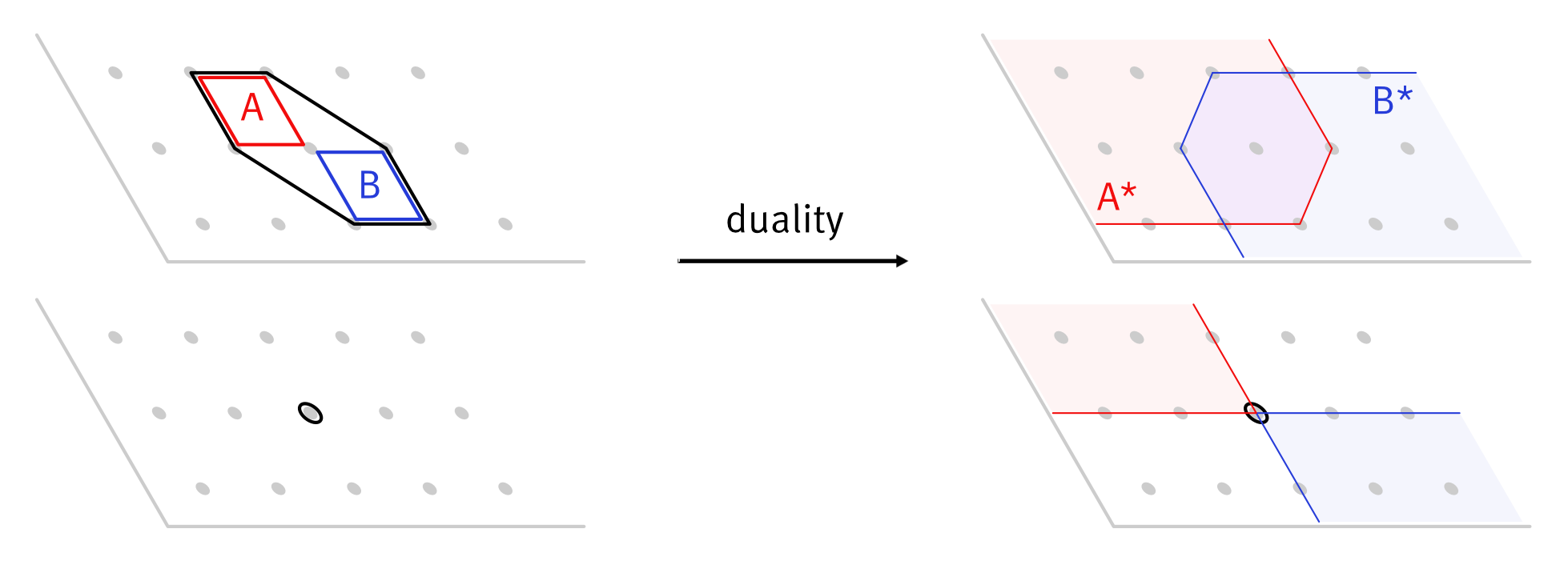}
  \caption{The dual picture of  the left-hand scaffolding in Figure~\ref{fig:mutating_scaffolding}.}
\label{fig:dP6_before_again}
\end{figure}

\begin{figure}[htb!]
  \centering
  \includegraphics[width=\textwidth]{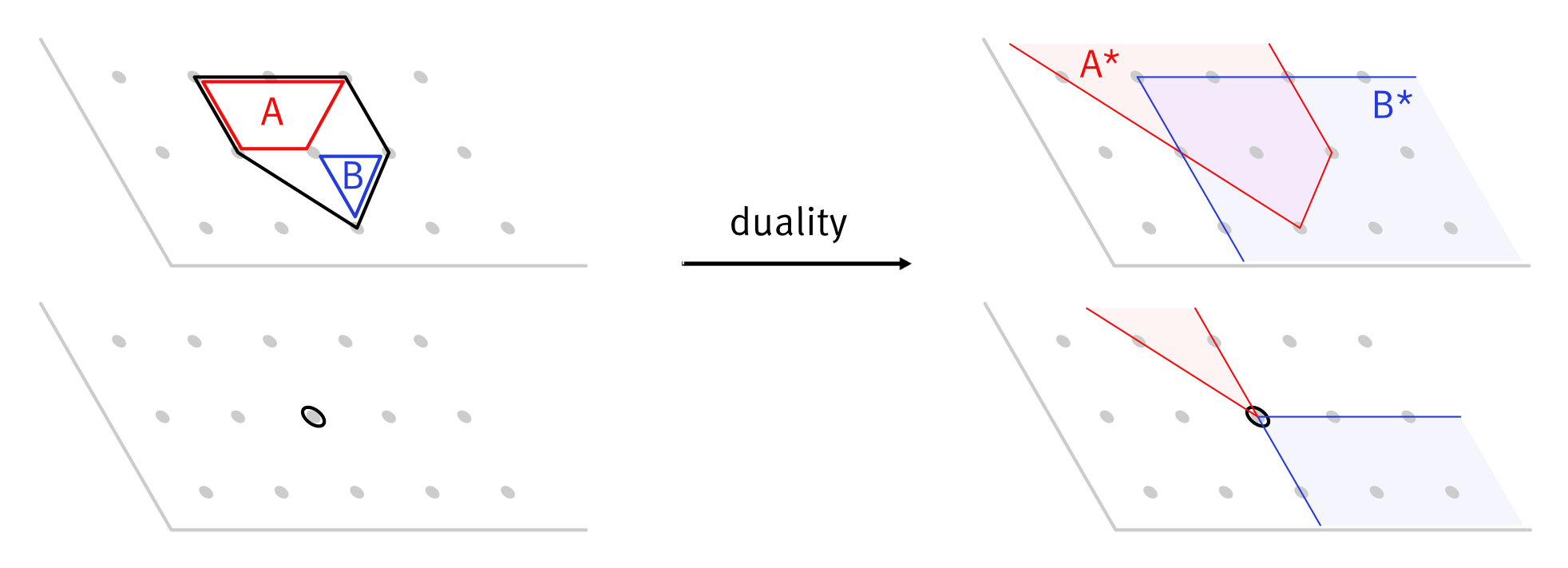}
  \caption{The dual picture of  the right-hand scaffolding in Figure~\ref{fig:mutating_scaffolding}.}
\label{fig:dP6_after}
\end{figure}

\noindent yields Figure~\ref{fig:dP6_after}. As before, on the left-hand side of Figure~\ref{fig:dP6_after} the scaffolding is placed at height~$1$ in $N_\QQ \oplus \QQ$, with the struts labelled as $A$ and~$B$. 
The corresponding cones $C_A$ and $C_B$ in $M_\QQ\oplus \QQ$ are shown on the right-hand side of Figure~\ref{fig:dP6_after}: $C_A$ is the cone over the dual polyhedron $A^*$, placed at height~$1$ in $M_\QQ\oplus \QQ$, and similarly for $C_B$. The tail cones $T_{A^*}$ of $A^*$ and $T_{B^*}$ of $B^*$ are shown at height zero: these are faces of $C_A$ and $C_B$ respectively. The shape $Z$ can be found by projecting the facets of $C_A$ and $C_B$ onto the height-zero slice in $M_\QQ \oplus \QQ$, where we see the fan of the Hirzebruch surface $Z = \FF_1$. 

Mutation here is the piecewise-linear transformation of $M_\QQ \oplus \QQ$ given by
\begin{equation}
\label{eq:mutation_PL}
  (x,y,z) \mapsto 
  \begin{cases}
    (x,y-x,z),&\text{ if }x < 0;\\
    (x,y,z),&\text{ if }x \geq 0.
  \end{cases}
\end{equation}
This maps the right-hand side of Figure~\ref{fig:dP6_before_again} to the right-hand side of Figure~\ref{fig:dP6_after}. One could also apply the definition of $N$-side mutation from~\cite{ACGK} directly to the struts in the left-hand side of Figure~\ref{fig:dP6_before_again}; note that in loc.~cit.~the polytope being mutated is not required to be Fano, or even to contain the origin in its interior. This yields the struts shown in the left-hand side of Figure~\ref{fig:dP6_after}.

Since in this example the shape $Z = \PP^1 \times \PP^1$ is a toric surface there is an alternative, and more geometric, description of its mutations which makes contact with the work of Gross--Hacking--Keel~\cite{GHK2}. A mutation of such a $Z$ is given by fixing a morphism $\pi \colon Z \rightarrow \PP^1$ and making an \emph{elementary transformation}\footnote{That is, blow up a point on one of the two torus invariant sections and contract the strict transform of the fibre containing this point.} of this $\PP^1$ bundle. In this case the mutation takes $Z$ to the Hirzebruch surface $\FF_1$. In general the fan determined by $Z$ undergoes a piecewise linear transformation $T$ which fixes the rays corresponding to the torus invariant sections of $\pi$. In this case $T$ is the restriction of \eqref{eq:mutation_PL} to the height-zero slice $z=0$.
\end{eg}

Turning now to nef partitions, we first extend the definition of nef partition to the setting of Fano toric complete intersections and then show that scaffolding offers a substantial generalisation of this new notion. We begin by recalling the basic definition and main results~\cite{B93,BN08}.

\begin{dfn}
Given a lattice $N$ and a reflexive polytope $\Delta \subset N_\QQ$, a \emph{nef partition of length~$r$} is a partition $E_1 \cup \ldots \cup E_r$ of the set $\V{\Delta}$ of vertices of $\Delta$ such that there are $\Sigma[\Delta]$-piecewise linear functions $\phi_i$ satisfying $\phi_i(v) = 1$ if $v \in E_i$ and $\phi_i(v) = 0$ otherwise. We write $\phi := \phi_1 + \ldots + \phi_r$.
\end{dfn}

\noindent A nef partition defines a set of nef divisors $D_i = \sum_{\rho \in E_i}{D_\rho}$ on $X_\Delta$ such that $\sum_{i=1}^{r}{D_i} = -K_{X_\Delta}$; thus a general section of the bundle $\bigoplus_{i=1}^{r}\cO(D_i)$ is a Calabi--Yau variety.

From the dual perspective, a nef partition is a Minkowski decomposition
\[
\Delta^* = \nabla_1 + \ldots + \nabla_r
\]
where the polytopes $\nabla_i$ are the polyhedra of sections of the line bundles $\cO(D_i)$, together with points $p_i \in \nabla_i$ for each $1\leq i\leq r$ such that $\sum_i{p_i} = 0$. The points $p_i$ themselves may be interpreted as the torus invariant divisors $D_i$, which determine unique sections of the bundles $\cO(D_i)$. More explicitly, the polytopes $\nabla_i$ are 
\begin{align*}
  \nabla_i := \{n \in N_\QQ \mid \text{$\langle n , m\rangle \geq \phi_i(m)$ for any  $m \in M_\QQ$}\}, && 1\leq i\leq r.
\end{align*}

In the case of a \emph{Fano} complete intersection we can make a directly analogous definition:

\begin{dfn}
Let $Y$ be the toric variety defined by a fan $\Sigma_Y$, and consider a partition of the rays $\Sigma_Y(1)$ into subsets $E_i$,~$1 \leq i \leq r$, and $F$. Let $D_i$ be the torus invariant divisor corresponding to the set $E_i$ and let $D_F$ be the torus invariant divisor corresponding to $F$. The partition is a \emph{Fano nef partition} if:
\begin{enumerate}
\item the divisor $D_F$ is ample; and
\item each of the divisors $D_i$ is nef.
\end{enumerate}
\end{dfn}

\noindent Note that since $D_F$ is ample and the divisor $\sum_{i=1}^{r} D_i$ is nef, the divisor $-K_{Y} = D_F + \sum_{i=1}^{r}D_i$ is ample, that is, $Y$ is a Fano toric variety.

\begin{lem}
The rays of $\Sigma_Y$ in the set $\bigcup_{i=1}^{r}E_i$ generate a Gorenstein cone of the fan $\Sigma_Y$. 
\end{lem}
\begin{proof}
Since $D_F$ is ample the stability condition defining $Y$ is covered by the divisor classes in $F$, and so the complement of these rays define a cone $\sigma$ in the fan $\Sigma_Y$. Note that $\Sigma_Y(1)\backslash F$ is precisely the set $\bigcup_{i=1}^{r}E_i$. Moreover, since each divisor $D_i$ is nef there is a function $\phi_i$ which is linear on $\sigma$ and evaluates to one on each of the ray generators of $E_i$ and to zero on all other ray generators of the cone $\sigma$. The sum $\phi$ of the $\phi_i$ defines a linear function on $\sigma$ evaluating to one on every generator, which implies that $\sigma$ is a Gorenstein cone.
\end{proof}

Consider a Fano polytope $P \subset N_\QQ$ and a scaffolding $S$ of $P$ with shape $Z = \prod_{i=1}^{k}{\PP^{a_i}}$. Lemma~\ref{lem:scaffolding_from_poly} and Theorem~\ref{thm:embedding} imply that these data determine a toric variety $Y_S$, divisors $D_1,\ldots,D_r$, on $Y_S$ whose linear systems define a Fano toric complete intersection, and a Laurent polynomial $f_S$ with $P = \Newt(f_S)$. Write $\Sigma_{Y_S}$ for the fan of $Y_S$, $E_i$ for the subset of the rays of $\Sigma_{Y_S}$ determined by $D_i$, and $F$ for the set $\Sigma_{Y_S}(1) \setminus \bigcup_{i=1}^{r} D_i$. If the divisors $D_i$ of $Y_S$ are nef, then $F \cup E_1 \cup \cdots \cup E_r$ is a Fano nef partition. Furthermore if $Y_S$ is $\QQ$-factorial, then the Laurent polynomial $f_S$ is mirror dual to the complete intersection defined by the vanishing of a general section of $\bigoplus_{i=1}^{r} \cO(D_i)$. Conversely, a Fano nef partition for which the rays in $\bigcup_{i=1}^{r}E_i$ span a smooth cone determines a scaffolding of a Fano polytope with shape $Z$ equal to a product of projective spaces.

\begin{rem}
The condition that the $D_i$ are nef is much stronger than it appears. In general $Y_S$ is far from being $\QQ$-factorial, in which case there is no reason for the $D_i$ to lie in $\QQ$-Cartier divisor classes. After making a small resolution of $Y_S$ it is reasonable to then expect the $D_i$ to be nef divisors, but we then usually lose the conclusion of Theorem~\ref{thm:embedding}.
\end{rem}

\begin{rem}
Recall that the ray generators of the fan $\Sigma_{Y_S}$ lie in $\Div_{T_{\bar{M}}}(Z)\oplus N_U$.
The set $E_i$ in the nef partition above is given by the $a_i+1$ divisors of the $i$th factor $\PP^{a_i}$ of the shape $Z = \prod_{i=1}^{k} \PP^{a_i}$. In particular, therefore, $E_i$ spans a smooth cone in $\Sigma_{Y_S}$.
This suggests a further generalisation of the notion of scaffolding in which the cone generated by the standard basis in $\Div_{T_{\bar{M}}}(Z)$ is replaced by an arbitrary Gorenstein cone. This is the most natural setting from the point of view of nef partitions: it would allow us to treat a broader class of toric complete intersections. We chose here, however, to pursue the alternative generalisation where the shape $Z$ need no longer be the product of projective spaces, as this allows us to describe embeddings of toric varieties that are very far from complete intersections. It would be very interesting to see if these ideas can be translated back to the Calabi--Yau setting, and whether they give access to more general embeddings of Calabi--Yau manifolds in toric varieties.
\end{rem}

Batyrev--Nill have determined necessary and sufficent conditions for a polytope to admit a nef partition~\cite{BN08}, based on certain \emph{Cayley cones} associated to a Minkowski decomposition of $\Delta^*$. 

\begin{dfn}
\label{dfn:cayley}
Given polytopes $\nabla_1,\ldots,\nabla_r$ in $N_\QQ$ the \emph{Cayley polytope} of length $r$ is
\[
\nabla_1\star \cdots \star \nabla_r := \conv{\nabla_1+e_1,\ldots, \nabla_r+e_r} \subset N_\QQ\times \QQ^r.
\]
The \emph{Cayley cone} is the cone
\[
\QQ_{\geq 0}(\nabla_1\star \cdots \star \nabla_r) = \QQ_{\geq 0}(\nabla_1+e_1) + \ldots +  \QQ_{\geq 0}(\nabla_r+e_r).
\]
\end{dfn}

\begin{pro}[\!\!\!{\cite[Proposition~3.6]{BN08}}]
Given a reflexive polytope $\Delta$ and a Minkowski decomposition 
\[
\Delta^* = \nabla_1 + \ldots + \nabla_r
\]
the following conditions are equivalent:
\begin{enumerate}
\item  the dual of the Cayley cone is a reflexive Gorenstein cone of index $r$ that can be realised as the Cayley cone of $r$ polytopes;
\item  the Cayley polytope $\nabla_1\star \cdots \star \nabla_r$ is a Gorenstein polytope of index\footnote{That is, a polytope $P$ such that $rP$ is reflexive, possibly after translation.} $r$ containing a special
$(r - 1)$-simplex (see~\cite{BN08});
\item the given Minkowski decomposition is a nef partition, that is, there are points $p_i \in \nabla_i$ for each $1\leq i\leq r$ such that $\sum_i{p_i} = 0$.
\end{enumerate}
\end{pro}

Given any scaffolding $S$ of a Fano polytope $P$, we can produce a large number of polytopes $\tilde{P}$ which project to $P$ using Cayley product-type constructions. For any lattice $L$ and any set of lattice vectors $R = \{r_s \in L \mid s \in S\}$, the polytope
\[
\tilde{P}_R := \conv{(P_D+\chi) + r_s \mid s = (D,\chi) \in S} \subset (N \oplus L)\otimes_\ZZ \QQ
\]
admits a projection to $P$, induced by the projection $N \oplus L \rightarrow N$. The scaffolding $S$ determines a canonical such polytope, given by setting 
\begin{align*}
  L = \Pic(Z), && R = \{\cO(D) \in \Pic(Z)\mid (D,\chi) \in S  \}.
\end{align*}
We denote this polytope $\tilde{P}_R$ by $\tilde{P}$. In the case where the shape $Z$ is a product of projective spaces, there is a natural choice of coefficents on the integral points of $\tilde{P}_R$ (for any $R$) that defines a Laurent polynomial with Newton polytope $\tilde{P}_R$ which projects to $f_S$.

 Given a scaffolding which defines a Fano nef partition we can describe both the toric ambient space $Y_S$ and the Laurent polynomial $f_S$ determined by $S$ in terms of Cayley products.

\begin{dfn}
\label{dfn:P_S}
Fix a Fano polytope $P$ and a scaffolding $S$ of $P$ with shape $Z = \prod_{i=1}^{k}{\PP^{a_i}}$ which determines a Fano nef partition of the toric ambient space $Y_S$. Define the polytope
\[
P_S := \conv{\{(e_i,0) \mid i \in \Sigma_Z(1)\} \cup S } \subset \tilde{N} = \Div_{T_{\bar{M}}}(Z) \oplus N_U.
\]
\end{dfn}

\noindent The toric variety defined by the spanning fan of $P_S$ is $Y_S$. Furthermore  the polytopes $\tilde{P}$ and $P_S$ are related by mutation. To describe this mutation we fix a boundary divisor $v_i$ of each projective space factor of $Z$. The divisors $v_i$ generate the kernel of a projection $\pi \colon \tilde{N} \rightarrow N$ and hence determine an isomorphism $\tilde{N} \rightarrow N\oplus \Pic(Z)$. Let $\tilde{P}_1$ denote the convex hull of $\tilde{P}$ and the set 
\[
\{\pi^\star_i\cO(1) \mid 1\leq i\leq r\} \subset  \{0\}\times \Pic(Z).
\]
A mutation of $\tilde{P}_1$ (or indeed of any other lattice polytope in $\tilde{N}_\QQ$) is determined by a \emph{weight vector }$w \in \tilde{M}$ and a polytope, the \emph{factor}, $F \subset w^\bot$. We fix a sequence of mutations indexed by $[r]$ by specifying their weight vectors $w_i$ and factors $F_i$, as follows:
\begin{enumerate}
\item let $w_i \in \tilde{M}$ be $-f^\star_i$, where $f^\star_i$ is the $i$th element of the basis dual to $\{v_1,\ldots,v_r\} \subset \tilde{N}$;
\item let $F_i$ be the the convex hull of the $(a_i+1)$ elements of the standard basis of $\Div_{T_{\bar{M}}}(Z)$ corresponding to the $i$th projective space factor in $Z$.
\end{enumerate}
The polytope obtained by applying the given sequence of mutations (in any order) to $\tilde{P}_1$ is~$P_S$.

\begin{eg}
\label{ex:dP4_polynomial}
We verify this in a simple example. Let $P$ and $S$ be the Fano polygon and scaffolding shown:

\begin{center}
\begin{tikzpicture}[transform shape]
\begin{scope}
\clip (-1.3,-1.3) rectangle (1.3cm,1.3cm); 
\draw[line width=0.5mm, black] (-1,1) -- (1,1) -- (1,-1) -- (-1,-1) -- (-1,1); 
\draw[line width=0.5mm, red] (-0.9,0.9) -- (0.9,0.9) -- (0.9,-0.9) -- (-0.9,-0.9) -- (-0.9,0.9) -- (0.9,0.9); 
\foreach \x in {-7,-6,...,7}{                           
    \foreach \y in {-7,-6,...,7}{                       
    \node[draw,shape = circle,inner sep=1pt,fill] at (\x,\y) {}; 
    }
}
 \node[draw,shape = circle,inner sep=4pt] at (0,0) {}; 
\end{scope}
\end{tikzpicture}
\end{center}
The shape here is $Z = \PP^1 \times \PP^1$. The Laurent polynomial associated to this scaffolding is
\[
f_S = \frac{(1+x)^2(1+y)^2}{xy}.
\]
Applying Algorithm~\ref{alg:laurent_inversion} to the scaffolding $S$ we obtain the toric variety $Y_S = \PP^4$ and an embedded toric degeneration of the del~Pezzo surface $dP_4$ to the surface $X_P$. The polytope $\tilde{P}_1$ is the Newton polytope of the polynomial
\[
g_S = z_1 + z_2 + \frac{(1+x)^2(1+y)^2}{xyz^2_1z^2_2}.
\]
Recall that the divisor $D$ defining the (unique) element $(D,0)$ of $S$ is a section of the line bundle $\cO(2,2) \in \Pic(\PP^1\times \PP^1)$. Mutating $g_S$ as described, we obtain the Laurent polynomial
\[
h_S = z_1(1+x) + z_2(1+y) + \frac{1}{xyz^2_1z^2_2}.
\]
The Newton polytope of $h_S$ is isomorphic to the Newton polytope of the polynomial
\[
f_{\PP^4} = x_1+x_2+x_3+x_4+\frac{1}{x_1x_2x_3x_4},
\]
that is, to the polytope $P_S$.
\end{eg}

Both of the scaffoldings described in Example~\ref{ex:dP6} arise from Fano nef partitions. Example~\ref{ex:dP6_mutated} shows that this property is not preserved under mutation of scaffoldings, whereas the Cayley polytope $\tilde{P}$ always exists. Thus the polytope $\tilde{P}$ associated to a scaffolding $S$ of $P$ is a natural generalisation of the notion of nef partition. 

\section{Amenable Collections and Towers of Projective Bundles}
\label{sec:torus_charts}

Theorem~\ref{thm:embedding} asserts that any scaffolding of a polytope $P$ determines an embedding of the toric variety $X_P$ into an ambient toric variety $Y$. Lemma~\ref{lem:scaffolding_from_poly} tells us that the Laurent polynomials obtained via the Przyjalkowski method encode enough data to reconstruct $X_P$ as a complete intersection, via a scaffolding on $P$ with shape a product of projective spaces. In fact the Przyjalkowski method can be generalised via the use of \emph{amenable collections subordinate to a nef partition}, introduced by Doran--Harder in~\cite{Doran--Harder}. These allow one to consider both more general toric complete intersection models for $X_P$ and more general Laurent polynomial mirrors $f$. In this section we show that these embeddings and Laurent polynomials are determined by scaffoldings of $P$ with a shape which is a tower of projective space bundles, rather than a product of projective spaces; in particular we see that our Laurent inversion construction (which allows the shape $Z$ to be any toric variety) generalises the methods of~\cite{Doran--Harder}.

Suppose, as before, that we have:
\begin{equation}
\label{eq:LG_data}
  \begin{minipage}{0.93\linewidth}
    \begin{enumerate}
    \item orbifold GIT data $\Theta = (K;\LL;D_1,\ldots,D_R;\omega)$;
    \item a convex partition with basis $B, S_1, \ldots, S_k, U$ for $\Theta$; and
    \item a choice of elements $s_i \in S_i$ for each $1\leq i\leq k$.
    \end{enumerate}
  \end{minipage}
\end{equation}
Let $Y$ be the corresponding toric orbifold,  let $X \subset Y$ denote the complete intersection defined by a regular section of the vector bundle $\bigoplus_i L_i$ and, following the notation used in~\S\ref{sec:laurent_inversion}, let $\widetilde{N}$ denote the ray lattice of $Y$. Following~\cite{Doran--Harder}, an \emph{amenable collection} subordinate to the partition $S_1, \ldots, S_k$ is a collection of vectors $w_1,\ldots,w_k$ that satisfies:
\begin{equation}
\label{eq:weight_vectors}
  \begin{minipage}{0.93\linewidth}
    \begin{enumerate}
    \item $\langle w_i , \rho_j \rangle = -1$ for all $j \in S_i$ and all $i$;
    \item $\langle w_i , \rho_j \rangle = 0$ for all $j \in S_l$ such that $l < i$ or $j \in U$ and all $i$;
    \item $\langle w_i, \rho_j \rangle \geq 0$ for all $j \in S_l$ such that $l > i$ and all $i$.
    \end{enumerate}
  \end{minipage}
\end{equation}
\begin{rem}
The condition $\langle w_i ,\rho_j\rangle = 0$ for $j \in U$ stems from the particular form of the algorithm used in~\S\ref{sec:forward}. There is a more general form of this algorithm in which this condition may be dropped.
\end{rem}

\noindent An amenable collection determines both a toric section of the bundle $\bigoplus_{1\leq i\leq k}L_i$, and so a toric degeneration of $X$, and a Laurent polynomial mirror $f$. These constructions are both explained in detail in~\cite{Doran--Harder}.

\begin{pro}
\label{pro:amenable}
An amenable collection determines and is determined by a tower of projective space bundles $Z$. Furthermore, given an amenable collection subordinate to a nef partition, the toric degeneration of $X$ to $X_P$ constructed in~\cite{Doran--Harder} is equal to the toric embedding determined by Theorem~\ref{thm:embedding} from a scaffolding of $P$ with shape $Z$.
\end{pro}
\begin{proof}
The toric embedding $X_P \hookrightarrow Y$ determined by an amenable collection has the following straightforward description in terms of the Cox co-ordinates of $Y$~\cite[Proposition~2.7]{Doran--Harder}. For each $1\leq i\leq k$ consider the binomial equation in Cox co-ordinates
\[
\prod_{j \in S_i}{x_j} - \prod_{j \notin S_i}{x_j^{\langle w_i, \rho_j\rangle}} = 0.
\]
\noindent The toric variety cut out by all of these equations is a toric degeneration of $X$. 

From an amenable collection we define $Z$ inductively, starting from a point $Z_0$. For each $1\leq j\leq k$ we define a toric variety $Z_j$ and a $\PP^{|S_j|-1}$ bundle $\pi_j \colon Z_j \rightarrow Z_{j-1}$. Each $Z_j$ is the projectivisation of a split vector bundle, and so is determined by a collection of line bundles on $Z_{j-1}$. First we specify line bundles $L_{m,n}$ for all $n \in S_j$ and $m < j$ recursively by setting
\[
L_{m,n} := \pi^\star_m(L_{m-1,n})\otimes \cO(-\langle w_m,\rho_n \rangle).
\]
Here $\cO(-1)$ is the tautological line bundle on the projective space fibration $\pi_j$ and $L_{0,n} := \cO$. Define $\pi_j$ to be the projectivisation of
\[
\PP_{Z_{j-1}}\left(\bigoplus_{n \in S_j} L_{j-1,n}\right)
\]
\noindent and define $Z := Z_k$. By construction the variety $Z$ is toric, and we can easily write down a generating set for the relations between rays of the fan of $Z$. Indeed, writing $z$ for the number of rays of $Z$, there is a partition of $[z]$ into $k$ sets of sizes $|S_1|,\ldots,|S_k|$ determined by the iterated bundle structure of $Z$. For each $1\leq i\leq k$ there is a relation $\sum_{j=1}^{z}{\alpha_{i,j}\rho_j}$ where $\alpha_{i,j} = -\langle \rho_j,w_i \rangle$. Note that the value of $-\langle \rho_j,w_i \rangle$ is positive only if $j \in S_i$, in which case it is equal to $1$. 

Recall that a scaffolding with shape $Z$ defines an embedding of lattices $N \rightarrow \widetilde{N} = \Div_{T_{\bar{M}}}(Z)\oplus N_U$. The relations described in the previous paragraph define hyperplanes in the lattice $\Div_{T_{\bar{M}}}(Z)$ and thus on $\widetilde{N}$. However any element $w$ in the dual lattice to $\widetilde{N}$ defines a binomial in Cox co-ordinates:
\[
\prod_{\rho\text{\ s.t.\ }\langle w,\rho \rangle > 0 }\!\!{x^{\langle w,\rho \rangle}_\rho} - \!\!\prod_{\rho\text{\ s.t.\ }\langle w,\rho \rangle < 0 }\!\!{x^{-\langle w,\rho \rangle}_\rho}.
\]
Evidently these binomials are precisely those defining $X_P$ as a subvariety of $Y$. Thus the system of binomials determined by an amenable collection is also determined by a scaffolding $S$ with shape $Z$, obtained by fixing the struts of $S$ (nef divisors on $Z$) via the projection $\widetilde{N} \rightarrow \Div_{T_{\bar{M}}}(Z)$. 
\end{proof}

\begin{rem}
This result is compatible with Mirror Symmetry: an amenable collection defines a Laurent polynomial $f$ much as in~\S\ref{sec:forward}, so that $f$ is the sum of terms $x_i$ whose Newton polyhedra are nef divisors on a tower of projective space bundles $Z$. Thus we can determine $Y$ and the toric embedding of $X_P$ from this Laurent polynomial $f$ and its scaffolding.
\end{rem}

\begin{eg}
A del~Pezzo surface $X_4$ of degree 4 is a $(2,2)$ complete intersection in $\PP^4$. Using the methods discussed in~\S\ref{sec:forward} one can construct a toric degeneration of this del~Pezzo surface with central fibre
\begin{align*}
  & x_0^2 - x_1x_2 = 0, && x_0^2 - x_3x_4= 0,
\end{align*}
\noindent where $x_0,\ldots,x_4$ are the homogeneous co-ordinates on $\PP^4$. Using amenable collections we now describe another toric degeneration of $X_4$. Let $\widetilde{N} \cong \ZZ^4$ be the ray lattice of $\PP^4$, and fix a convex partition with basis by setting $B = \{1\}$, $S_1 = \{2,3\}$, $S_2 = \{4,5\}$, and $U = \varnothing$. Choose an amenable collection $\{w_1, w_2\}$ in $M$ by setting
\begin{align*}
& w_1 = (-1,-1,0,2), 
&& w_2 = (0,0,-1,-1).
\end{align*}
The two equations defined by the $w_i$ are
\begin{align}\label{eq:binomial_quadrics}
& x^2_4 - x_1x_2 = 0, &&
x^2_0 - x_3x_4 = 0.
\end{align}
To compute the corresponding scaffolding we first need to determine $Z$. Following the proof of Proposition~\ref{pro:amenable} we see that $Z \cong \mathbb{F}_2 := \PP_{\PP^1}(\cO\oplus \cO(-2))$. The scaffolding is determined by taking rays of $Y$ not contained in the standard basis and viewing them as nef divisors on $\mathbb{F}_2$. In this case we have only the ray $(-1,-1,-1,-1)$, which corresponds to the toric boundary of $\mathbb{F}_2$. Consequently the scaffolding we obtain consists of a single triangle: see Figure~\ref{fig:scaff2}. The rays shown on Figure~\ref{fig:scaff2} are obtained by pulling back the fan of $\PP^4$ along the inclusion of the subspace of $\widetilde{N}$ annihilating both $w_1$ and $w_2$. In particular the toric variety defined by this fan is a quotient of the weighted projective plane $\PP(1,1,2)$ defined by the binomial quadrics \eqref{eq:binomial_quadrics} in $\PP^4$. This is an instance of Proposition~\ref{pro:anticanonical}, with shape $\FF_2$.
\begin{figure}
\begin{tikzpicture}[transform shape]
\begin{scope}
\clip (-1.3,-1.3) rectangle (1.3cm,3.3cm); 
\draw[line width=0.5mm, black] (-1,3) -- (-1,-1) -- (1,-1) -- (-1,3); 
\draw[line width=0.5mm, red] (-0.9,2.6) -- (-0.9,-0.9) -- (0.85,-0.9) -- (-0.9,2.6) -- (-0.9,-0.9);
\foreach \x in {-7,-6,...,7}{                           
    \foreach \y in {-7,-6,...,7}{                       
    \node[draw,shape = circle,inner sep=1pt,fill] at (\x,\y) {}; 
    }
}
 \node[draw,shape = circle,inner sep=4pt] at (0,0) {}; 
\end{scope}
\end{tikzpicture}
\caption{A scaffolding determined by an amenable collection.}
\label{fig:scaff2}
\end{figure}
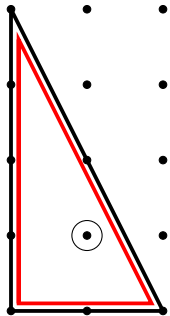
\end{eg}
\section*{Acknowledgements}
We thank Alessio Corti and Andrea Petracci for many useful conversations, and Cinzia Casagrande for an extremely helpful observation about the Fano manifold in~\S\ref{sec:new_4d}. TC was supported by ERC Starting Investigator Grant 240123, ERC Consolidator Grant 682603, and EPSRC Program Grant EP/N03189X/1. AK was supported by ERC Starting Investigator Grant 240123 and EPSRC Fellowship~EP/N022513/1. TP was supported by an EPSRC Prize Studentship and an EPSRC Doctoral Prize Fellowship. TC thanks the University of California at Berkeley for hospitality during the writing of this paper.

\appendix

\section{The Proof of Theorem~\ref{thm:embedding}}
\label{sec:embedding_proof}

Throughout this section we fix a Fano polytope $P$ together with a scaffolding $S$ of $P$ with shape $Z$. We show that $X_P$ is a toric subvariety of the ambient space $Y_S$ defined in~\S\ref{sec:laurent_inversion}, via the embedding of tori defined in the discussion following Theorem~\ref{thm:embedding}. We begin by constructing a polytope $Q_S$ defined by a polarisation of the toric variety $Y_S$.

\begin{dfn}
\label{dfn:ambient}
  Let $\tN$ denote the lattice $\Div_{T_{\bar{M}}}(Z)\oplus N_U$ and let $\tM$ denote the dual lattice. Denote the standard basis elements of $\Div_{T_{\bar{M}}}(Z) \cong \ZZ^k$ by $e_i$ for $1\leq i\leq k$. Define elements $\rho_s = (-D,\chi) \in \tN$ for each $s = (D,\chi) \in S$. The polytope $Q_S$ is defined by
  \[
  Q_S := \left\{u \in \tM_\QQ \mid \langle u,e_i \rangle \geq 0\text{ and }\langle u,\rho_s \rangle \geq -1\text{ for all  }s \in S\text{ and }1\leq i\leq k\right\}.
  \] 
  We write $\Sigma_S$ for the normal fan of $Q_S$. 
\end{dfn}

\begin{rem}
	We assume throughout this section that, given an element $(D,\chi) \in S$, the strut $P_D + \chi$ contains a vertex of $P$. In fact $X_P$ embeds in $Y_S$ if and only if it embeds into $Y_{S'}$ where $S'$ is the scaffolding obtained by removing all struts which do not meet a vertex of $P$.
\end{rem}

\begin{rem}
It is elementary to check that the toric variety defined by $\Sigma_S$ is precisely $Y_S$. Indeed, the polarising class is precisely the one chosen in Algorithm~\ref{alg:laurent_inversion}.
\end{rem}

We now study the faces of $Q_S$ in more detail. Our first step is to introduce a polyhedral decomposition of $Q := P^*$. 

\begin{dfn}
	\label{dfn:strut_vertex}
	Let $\V{S}$ be the set of torus fixed points of $Z$, and observe there is a canonical bijection $\V{S} \rightarrow \V{P_{D'}}$ for an ample divisor $D'$,
	and a canonical surjection $\V{S} \rightarrow \V{P_D}$ for a nef divisor $D$ which we denote by $u \mapsto u^D$. We refer to $\V{S}$ as the set of vertices of the scaffolding $S$. Each element $u \in \V{S}$ defines a function $S \to N$, which we also denote by $u$, defined by setting $u\left((D,\chi)\right) = u^D + \chi$.
\end{dfn}

\begin{dfn}
	\label{dfn:U_vertices}
	Given $u \in \V{S}$ we define 
	\[
	\V{P,u} := \{v \in \V{P} \mid v = u(s)\text{ for some }s \in S \}.
	\]
\end{dfn}

\begin{dfn}
	\label{dfn:polyhedral_decomposition}
	Define a polyhedral decomposition of $Q$ by intersecting $Q$ with the fan $\Sigma_Z \times (N_U\otimes \QQ)$ defining the toric variety $Z\times T_{N_U}$. Maximal cells $C_u$ of this decomposition are indexed by elements $u \in \V{S}$.
\end{dfn}
\begin{rem}
	\label{rem:min_function}
	If we identify $\V{S}$ with the vertices of the polyhedron of sections of an ample divisor $D$ on $Z$ the chambers $C_u$ are precisely the maximal domains of linearity of the convex piecewise linear function 
	\[
	\min_{u \in \V{S}}\langle u^D, - \rangle \colon Q \rightarrow \QQ
	\]
	indexed by the vertex $u^D$ on which this function attains its minimum. If instead $D$ is nef then the analogous maximal domains of linearity are unions of chambers $C_u$.
\end{rem}
We next identify certain faces of $Q_S$ with images $\iota(C_u)$ of a piecewise linear function $\iota$.
\begin{dfn}
	\label{dfn:iota}
	Let $n = \dim M$. Define $\iota$ to be the inverse map to the restriction to $E \oplus N_U$ of the canonical projection $\tM_\QQ \rightarrow M_\QQ$,  where $E$ is the union of $n$-dimensional faces of the standard coordinate cone in $\Div_{T_{\bar{M}}}(Z)^\vee$ which project onto maximal dimensional cones of the fan of $Z$. 
\end{dfn}
The fact that $Z$ is $\QQ$-factorial ensures that $\iota$ is well defined. Note that $\iota$ is linear on each chamber $C_u \subset Q_\QQ$.
\begin{pro}
	\label{pro:iota_union}
	For each $u \in \V{S}$, the polytope $\iota(C_u)$ is a face of $Q_S$.
\end{pro}
\begin{proof}
	We first show that, for any $p \in C_u$:
	\begin{enumerate}
		\item $\langle e_i, \iota(p)\rangle = 0$ for some $1\leq i\leq k$;
		\item $\langle e_j, \iota(p)\rangle \geq 0$ for all $1\leq j\leq k$; and
		\item $\langle \rho_s, \iota(p)\rangle \geq -1$ for all $s \in S$.
	\end{enumerate}
	Fixing a point $p \in C_u$ the first two are obvious: $\iota(p)$ lies in the positive co-ordinate cone of $\tM_\QQ$ so the second condition is automatic, the first follows from the fact that $\iota(p)$ lies in the cone spanned by $n=\dim{M}$ of the standard coordinate vectors and hence in the hyperplane defined by $\langle e^\star_i,- \rangle$ for $e_i$ not among these $n$ vectors.
	
	Next we consider $\langle \rho_s, \iota(p)\rangle$; the map $\iota$ is linear on $C_u$ and let $\iota_u$ denote the linear extension of $\iota|_{C_u}$ to $M_\QQ$. We now compute $\langle \iota_u^\star\rho_s, p\rangle$. It is clear that $\iota_u(M_\QQ)$ is the span of $M_U$ together with co-ordinate vectors $e^\star_i$ in $\tM$ such that the divisor in $Z$ corresponding to $e_i$ meets the vertex $u$. By definition $\langle \rho_s, e^\star_i\rangle$ is the height of the supporting hyperplane of the polyhedron of sections $P_D$ where $s = (D,\chi)$. Thus $\iota_u^\star\rho_s$ is the vertex $u$ of $P_D+\chi$ and therefore lies in $P$. Since $\iota_u^\star\rho_s \in P$ and $p \in Q$, $\langle \iota_u^\star\rho_s, p\rangle \geq -1$.
	
	We have shown that $\iota(C_u)$ is contained in a face of $Q_S$, to show the reverse inclusion we need only to check that if $\langle \rho_s,m' \rangle \geq -1$ for $m' \in \iota_u(M)$ and $m'$ in the standard positive cone then $m' \in C_u$. However this also follows from the fact that $\iota_u^\star\rho_s$ is the vertex $u$ of $P_D+\chi$.
\end{proof}

The polytope $Q_S$ determines its normal fan $\Sigma_S$, which in turn determines a toric variety $Y_S$. We now prove that the pullback of the fan $\Sigma_S$ under the inclusion $M \rightarrow \tM$ is the spanning fan of the Fano polytope $P$.

\begin{lem}
	\label{lem:all_rays}
	The set of rays of $\Sigma_S$ is $\{\rho_s \mid s \in S\}\cup\{e_i \mid 1\leq i\leq k\}$. That is, all the rays used in Definition~\ref{dfn:ambient} to define $Q_S$ appear in the normal fan of $Q_S$.
\end{lem}
\begin{proof}
	Finding facets of $Q_S$ with normal direction $e_i$, $1\leq i\leq k$, is straightforward: intersecting $Q_S$ with a small ball $B$, so that $\langle \rho_s,p\rangle > -1$ for all $p \in B$, centered at the origin we obtain a smooth (not necessarily strictly convex) cone. The normal directions to the facets meeting the origin are precisely the co-ordinate vectors $e_i$.
	
	Now fix an element $s = (D,\chi) \in S$ and a vertex $v \in P$ which meets the strut $P_D+\chi$. Let $B'$ be a small ball around a point $\iota(p)$, where $p$ is a point in the interior of the facet $v^\star$ dual to the vertex $v$. Now recall from the proof of Proposition~\ref{pro:iota_union} that for any $s' \in S$ and $u \in \V{S}$
	\[
	\iota_u^\star\rho_{s'} = u(s')
	\]
	where $\iota_u^\star$ is the dual map to the linear map defined by restricting $\iota$ to $C_u$. Consequently  considering $\rho_{s'}$ as a function on $\iota(\partial Q)$ we see that $\rho_{s'}$ achieves its minimum, $-1$, precisely along facets $u(s)^\star$ for $u(s)$ a vertex of $P$. Therefore, taking a point $p'$ in the intersection of $B'$ with the hyperplane $\langle \rho_s,-\rangle = -1$ and the half spaces $\langle e_i,-\rangle > 0$ for $1\leq i\leq k$, by construction $p'$ lies on the facet with normal $\rho_s$.
\end{proof}

We require an explicit description of those cones in $\Sigma_S$ which intersect the image of $N_\QQ$ non-trivially. Fixing a face $E$ of $\partial P$ we first identify a cone in $\tN_\QQ$ which intersects $N_\QQ$ precisely in the cone generated by $E$.
\begin{dfn}
	Given a stratum $E \in \partial P$ let
	\[
	\V{S,E} := \{u \in \V{S} \mid u(s) \in E\text{ for some }s \in S\}.
	\]
	\noindent Let $C_E \subset \tN_\QQ$ be the cone generated by the following rays:
	\begin{enumerate}
		\item $\rho_s$, for $s=(D,\chi) \in S$ such that the strut $P_D+\chi$ meets $E$;
		\item $e_i$, the divisors of $Z$ which miss some $u \in \V{S,E}$. 
	\end{enumerate}
\end{dfn}

\begin{pro}
	\label{pro:good_cones}
	We have that $C_E \cap N_\QQ = \cone(E)$.
\end{pro}
\begin{proof}
	First we show that $\cone(E) \subset C_E \cap N_\QQ$. In fact we prove that every ray generator of $\cone(E)$ appears in $ C_E \cap N_\QQ$. By definition the ray generators of $\cone(E)$ are precisely the vertices of the face $E \subset P$. Choose such a vertex $v$. Writing $\theta$ for the inclusion $N \rightarrow \tN$, since $v = \iota_u^\star(\rho_s)$ for any $u$ such that $v \in \V{P,u}$ we see that $(\theta(v) - \rho_s) \in \ker(\iota_u^\star)$. However $\ker(\iota_u^\star)$ is the annihilator of $\iota_u(M_\QQ) \subset \tM$, which is generated by those divisors $e_i$ which miss the torus fixed point $u \in \V{S}$ which determined the linear map $\iota_u$. That is, we can express $v$ as a linear combination of ray generators of $\cone(E)$. Thus to prove that $v \in C_E \cap N_\QQ$ we now only need to prove that the divisor $\theta(v) - \rho_s$ is effective on $Z$. However the polyhedron of sections of this divisor is easily seen to be that of $-\rho_s$ translated so that the origin is identified with the $u$ vertex of $P_D$ (for $s = (D,\chi)$). Since this polyhedron contains the origin it must be the polyhedron of sections of an effective divisor. 
	
	To prove the reverse inclusion observe that it suffices to consider the case where $E$ is a vertex. Since the generators of $C_E$ different from $\rho_s$ generate the subspace $\ker(\iota_u^\star) \subset \tN$ they span a complement to $\theta(N)$, and therefore $\dim\left(C_E \cap N_\QQ\right) \leq 1$. Thus this intersection is contained in a ray which, by the first part, must be $\cone(E)$.
\end{proof}

We now conclude the proof of Theorem~\ref{thm:embedding} by showing that the cones $C_E$ appear in the fan~$\Sigma_S$. We show that the following diagram commutes:
\begin{equation}
  \begin{aligned}
\label{eq:diagram}
    \xymatrix{
      \widetilde{\Sigma}^{max}_S \ar[rr] \ar^{-\cap N_{\QQ}}[d] & & \V{Q_S} \\
      \Sigma^{max}_{P} \ar[rr] & & \V{Q} \ar^{\iota}[u].
    }
  \end{aligned}
\end{equation}
Here the horizontal arrows are the usual bijections between $k$-strata of a polytope and codimension-$k$ cones of its normal fan, and $\Sigma_P$ is the spanning fan of the polytope $P$. The right-hand vertical arrow is the piecewise linear map $\iota$ and the left-hand vertical map is intersection with the subspace $N_\QQ$, where we slightly abuse notation to make this map well-defined by restricting to those maximal cones in $\tN_\QQ$ which meet $N_\QQ$ in a maximal dimensional cone. Choose a maximal cone in $\Sigma_S$ which meets $N_\QQ$ in maximal dimensional cone $\sigma$; let $v$ be the corresponding vertex of $Q$. We need to check that the normal directions to the facets which meet $\iota(v) \in \V{Q_S}$ generate $C_\sigma$.
	
First consider the collection of struts $\{\rho_s \mid s \in S, \langle \rho_s, \iota(v) \rangle = -1\}$. Recall from the proof of Proposition~\ref{pro:iota_union} that given an element $u \in \V{S}$ and $w \in \V{v^\star}$ such that $w \in \V{P,u}$ we have that $\iota^\star_u(\rho_s) = w$. Now $v \in C_u$ for some $u \in \V{S}$ and observe that $\iota(v)=\iota_u(v)$ for any such $u \in \V{S}$, and  	$\langle \iota^\star_u(\rho_s), v \rangle = \langle u(s), v \rangle$, so the condition that $\langle \rho_s, \iota(v) \rangle = -1$ is equivalent to the condition that
\[
\langle u(s), v \rangle = -1
\]
for some $u \in \V{S}$ such that $v \in C_u$. That is, that $u(s) \in v^\star$ for some $u$. In other words, the set of struts $\{\rho_s \mid s \in S, \langle \rho_s, \iota(v) \rangle = -1\}$ is precisely the set of those struts which meet $v^\star$. Thus the collection of hyperplanes $\langle \rho_s,-\rangle = -1$ meeting $\iota(v)$ are precisely the generators $\rho_s$ of $C_{v^\star}$. Furthermore, by Lemma~\ref{lem:all_rays} these hyperplanes do all appear as facets of $Q_S$ that meet $v$. It remains to check that the generators $e_i$ of $C_{v^\star}$ appear as rays in the normal cone to $\iota(v)$.
	
	Consider the collection $\{e_i \mid 1\leq i\leq k, \langle e_i,\iota(v) \rangle = 0 \}$, and observe that by Lemma~\ref{lem:all_rays} each of these hyperplanes defines a facet of $Q_S$ that meets $v$. Recall that the elements $u \in \V{S}$ are in bijection with the maximal cones $\sigma_u$ in the fan $\Sigma_Z$ defining $Z$, and that by definition $\iota(\sigma_u)$ is the cone in $\tM$ whose ray generators are standard basis vectors corresponding to rays of $\sigma_u$. Choosing any $u$ such that $v \in C_u$ the functionals $e_i$ which vanish on the cone $\iota(\sigma_u)$ are precisely those elementary torus invariant divisors which miss the torus fixed point $u \in \V{S}$. Conversely $\iota(\cone(v))$ is the intersection of cones $\iota(\cone(\sigma_u))$ for $v \in C_u$, and so the annihilator of $\iota(\cone(v))$ is generated by the union of those standard basis vectors generating the annihilator of $\iota(\sigma_u)$. That is, if $e_i$ vanishes on $\iota(v)$ it vanishes on some $\sigma_u$.
	
	Thus the facets which meet $\iota(v)$ exactly generate the cone $C_{\cone(v^\star)}$. Proposition~\ref{pro:good_cones} now implies that diagram~\eqref{eq:diagram} commutes. Consider now the fan $\Sigma'_P$ obtained by pulling back the fan $\Sigma_S$ along the inclusion $\theta \colon N_\QQ \hookrightarrow \tN_\QQ$. Commutativity of diagram~\eqref{eq:diagram} guarantees that the cone over each face of $P$ occurs in $\Sigma'_P$. Since the spanning fan of $P$ is a complete fan, we have determined every cone of $\Sigma'_P$. This completes the proof of Theorem~\ref{thm:embedding}.

\bibliographystyle{plain}
\bibliography{bibliography}
\end{document}